\documentclass[a4paper]{amsart} 
\usepackage{amsmath,amsfonts,amssymb,graphicx,url,amsthm,hyperref}
\usepackage[left=3cm,right=3cm,top=3cm,bottom=3cm]{geometry}
\usepackage[nocompress]{cite}
\usepackage{xcolor}
\usepackage[noend]{algpseudocode}
\usepackage[ruled,vlined,norelsize,algosection]{algorithm2e}
\usepackage[T1]{fontenc}
\usepackage[mathlines]{lineno} 

\newcommand{\R}{\mathbb{R}}

\newcommand{\kernel}{\mathcal{K}}

\renewcommand{\d}{\operatorname{d}\!}
\newcommand{\bs}{\boldsymbol}
\DeclareMathOperator{\diam}{diam}

\DeclareMathOperator{\spn}{span}

\DeclareMathOperator\dist{dist}
\DeclareMathOperator\supp{supp}

\usepackage{caption}
\usepackage{subcaption}
\usepackage{tikz,pgfplots,pgfplotstable}
\usetikzlibrary{external}
\pgfplotsset{compat=newest}
\usetikzlibrary{arrows,snakes,backgrounds,positioning,calc,matrix}
\usepackage{graphicx,color,nicefrac} 
\usepackage{amssymb, amsmath, amsbsy,mathtools}
\usepackage{amsmath,amsfonts}
\usepackage[export]{adjustbox}
\newtheorem{theorem}{Theorem}[section]
\newtheorem{lemma}[theorem]{Lemma}
\newtheorem{corollary}[theorem]{Corollary}

\newtheorem{definition}[theorem]{Definition}

\newtheorem{remark}[theorem]{Remark}

\def\letters{a,b,c,d,e,f,g,h,i,j,k,l,m,n,o,p,q,r,s,t,u,v,w,x,y,z}
\def\Letters{A,B,C,D,E,F,G,H,I,J,K,L,M,N,O,P,Q,R,S,T,U,V,W,X,Y,Z}
\makeatletter
\@for \@l:=\Letters \do{%
  \expandafter\edef\csname\@l bb\endcsname{\noexpand\ensuremath{\noexpand\mathbb{\@l}}}%
  \expandafter\edef\csname\@l bf\endcsname{{\noexpand\bf \@l}}%
  \expandafter\edef\csname\@l cal\endcsname{\noexpand\ensuremath{\noexpand\mathcal{\@l}}}%
  \expandafter\edef\csname\@l eu\endcsname{\noexpand\ensuremath{\noexpand\EuScript{\@l}}}%
  \expandafter\edef\csname\@l frak\endcsname{\noexpand\ensuremath{\noexpand\mathfrak{\@l}}}%
  \expandafter\edef\csname\@l rm\endcsname{{\noexpand\rm \@l}}%
  \expandafter\edef\csname\@l scr\endcsname{\noexpand\ensuremath{\noexpand\mathscr{\@l}}}%
}
\@for \@l:=\letters \do{%
  \expandafter\edef\csname\@l bf\endcsname{{\noexpand\bf \@l}}%
  \expandafter\edef\csname\@l frak\endcsname{\noexpand\ensuremath{\noexpand\mathfrak{\@l}}}%
  \expandafter\edef\csname\@l scr\endcsname{\noexpand\ensuremath{\noexpand\mathscr{\@l}}}%
}
\makeatother
 \definecolor{shadecolor}{rgb}{0.6, 0.6, 0.6} 
  \definecolor{red}{rgb}{0,0,0} 
  \definecolor{darkgreen}{rgb}{0, 0.6, 0}
\newcommand{\isdef}{\mathrel{\mathrel{\mathop:}=}}
\newcommand{\defis}{\mathrel{=\mathrel{\mathop:}}}
\begin{document}
\title{Samplets: A new paradigm for data compression}
\author{Helmut Harbrecht}
\address{Helmut Harbrecht,
Departement f\"ur Mathematik und Informatik, 
Universit\"at Basel, 
Spiegelgasse 1, 4051 Basel, Switzerland.}
\email{helmut.harbrecht@unibas.ch}
\author{Michael Multerer}
\address{
Michael Multerer,
Euler Institute,
USI Lugano,
Via la Santa 1, 6962 Lugano, Svizzera.}
\email{michael.multerer@usi.ch}

\begin{abstract}
In this article, we introduce the concept of samplets 
by transferring the construction of Tausch-White wavelets 
\cite{TW03} to the realm of data. This way we obtain a
multilevel representation of discrete data which directly
enables data compression, detection of singularities
and adaptivity. Applying samplets to represent kernel 
matrices, as they arise in kernel based 
learning or Gaussian process regression, we end up with 
quasi-sparse matrices. By thresholding
small entries, these matrices are compressible to
\(\mathcal{O}(N\log N)\) relevant entries, where \(N\)
is the number of data points. This feature allows
for the use of fill-in reducing reorderings to obtain
a sparse factorization of the compressed matrices.
Besides the comprehensive introduction to samplets and
their properties, we present extensive numerical studies
to benchmark the approach.
Our results demonstrate that samplets mark a considerable
step in the direction of making large data sets
accessible for analysis.
\end{abstract}

\maketitle
\section{Introduction}\label{sec:intro}
Wavelet techniques have a long standing history in the field of data
science. Applications comprise signal processing, image analysis and
machine learning, see for instance
\cite{Chui,Dahmen,Daubechies,Mallat,Mallat2016}
and the references therein. Assuming a signal generated by some
function, the pivotal idea of wavelet techniques is the splitting of
this function into its contributions with respect to a
hierarchy of scales. Such a multiscale ansatz starts from an
approximation on a relatively coarse scale and successively resolves
details at finer scales. Hence, compression and adaptive
representation are inherently built into this ansatz. The transformation
of a given signal into its wavelet representation and the inverse
transformation can be performed with linear cost in terms of the degrees
of freedom.

Classically, wavelets are constructed by refinement relations and
therefore require a sequence of nested approximation spaces which are
copies of each other, except for a different scaling. This restricts the
concept of wavelets to structured data. Some adaption of the general
principle is possible in order to treat intervals, bounded domains and
surfaces, compare \cite{Alp93,DKU,HS,PSS97,Quak,STE} for example. The 
seminal work \cite{TW03} by Tausch and White overcomes this obstruction by
constructing wavelets as suitable linear combinations of functions at a
given fine scale. In particular, the stability of the resulting basis,
which is essential for numerical algorithms is guaranteed by
orthonormality.

In this article, we take the concept of wavelets to the next level and
consider discrete, unstructured data. To this end, we modify the construction of
Tausch and White and construct a multiscale basis which
consists of localized and discrete signed measures. Inspired by the term
wavelet, we call such signed measures \emph{samplets}. Samplets can be
constructed such that their associated measure integrals vanish for
polynomial integrands. If this is the case for all polynomials of total
degree less or equal than \(q\), we say that the samplets have
\emph{vanishing moments} of order $q+1$. We remark that lowest order
samplets, i.e.\ \(q=0\), have been considered earlier for data
compression in \cite{RE11}. Another concept for constructing multiscale
bases on data sets are \emph{diffusion wavelets}, which employ a
diffusion operator to construct the multiscale hierarchy, see
\cite{CM06}. In contrast to diffusion wavelets, however, the
construction of samplets is solely based on discrete structures and can
always be performed with linear cost for a balanced cluster tree, even
for non-uniformly distributed data.

When representing discrete data by samplets, then, due to the vanishing
moments, there is a fast decay of the corresponding samplet coefficients
with respect to the support size if the data are smooth. This
straightforwardly enables data compression. In contrast, non-smooth
regions in the data are indicated by large samplet coefficients. This,
in turn, enables singularity detection and extraction. Furthermore, the
construction of samplets is not limited to the use of polynomials.
Indeed, it is easily be possible to adapt  the construction to other
primitives with different desired properties.

The second application of samplets we consider is compression of kernel
matrices, as they arise in kernel based machine learning and scattered
data approximation, compare \cite{Fasshauer2007,HSS08,Rasmussen2006,%
Schaback2006,Wendland2004,Williams1998}. Kernel matrices are typically
densely populated, since the underlying kernels are nonlocal.
Nonetheless, these kernels are usually \emph{asymptotically smooth},
meaning  that they behave like smooth functions apart from the diagonal.
A discretization of an asymptotical smooth kernel with respect to a
samplet basis with  vanishing moments results in quasi-sparse kernel
matrices, which means that they can be compressed such that only a
sparse matrix remains, compare \cite{BCR,DHS,DPS,PS,SCHN}. Especially,
it has been demonstrated in \cite{HM} that nested dissection, see
\cite{Geo73,LRT79}, is applicable in order to obtain a fill-in reducing
reordering of the matrix in the standard form. This reordering in turn 
allows for the rapid factorization of the system matrix by the Cholesky
factorization without introducing additional errors. \textcolor{red}{This 
is in contrast to the approximate computation of the Cholesky factorization 
with respect to the so-called non-standard form of operators or
by $\mathcal{H}$-matrices which has been proposed earlier, compare
\cite{GBD98,HACK}.}

The asymptotic smoothness of the kernels is also exploited by cluster
methods, like the fast multipole method, see \cite{GR,RO,YBZ04} and
particularly \cite{MXTY+2015} for high-dimensional data. However,
these methods do not allow for the direct and exact factorization, which
is for example advantageous for the simulation of Gaussian
random fields. A further approach, which is more in line of the present
work, is the use of \emph{gamblets}, see \cite{Owh17}, for the
compression of the kernel matrix, cp.\ \cite{SSO21}. Different from
the discrete construction of samplets with vanishing
moments, the construction of gamblets is adapted to an underlying
pseudo-differential operator and basis functions need to be truncated
in order to obtain localized supports,
while localized supports are automatically obtained
by the samplet construction.

As samplets are directly constructed with respect to a discrete data
set, their applications are manifold. Within this article, we
particularly consider time-series data, image data, kernel matrix
representation and the simulation of Gaussian random fields as examples.
We remark, however, that we do not claim to have invented a new method for 
high-dimensional data approximation. The current construction is based
on total degree polynomials and is hence not dimension robust, thus
limited to data of moderate dimension. Even so, we believe that samplets
provide most of the advantages of other approaches for scattered data,
while being easy to implement. Especially, most of the algorithms
available for wavelets with vanishing moments are
transferable.

The rest of this article is organized as follows. In
Section~\ref{section:Samplets}, the concept of samplets is introduced.
The subsequent Section~\ref{sct:construction} is devoted to the actual
construction of samplets and to their  properties. The change of basis
by means of the discrete samplet transform is the topic of
Section~\ref{sec:FST}. In Section \ref{sec:Num1}, we demonstrate the
capabilities of samplets for data compression and smoothing for data in
one, two and three dimensions. Section~\ref{sec:kernelCompression} deals 
with the samplet compression of kernel matrices. Especially, we also
employ an interpolation based \(\Hcal^2\)-matrix approach in order to
efficiently assemble the compressed kernel matrix. Corresponding
numerical results are then presented in Section \ref{sec:Num2} for up
to four dimensions. Finally, in Section~\ref{sec:Conclusion}, we state
concluding remarks. 

\section{Samplets}
\label{section:Samplets}
Let \(X\isdef\{{\bs x}_1,\ldots,{\bs x}_N\}\subset\Omega\) 
denote a set of points within some region \(\Omega\subset\Rbb^d\).
Associated to each point \({\bs x}_i\), we introduce
the Dirac measure 
\[
\delta_{{\bs x}_i}({\bs x})\isdef
\begin{cases}
1,&\text{if }{\bs x}={\bs x}_i\\
0,&\text{otherwise}.
\end{cases}
\]
With a slight abuse of notation, we also introduce the
point evaluation functional
\[
(f,\delta_{{\bs x}_i})_\Omega=\int_\Omega
f({\bs x})\delta_{{\bs x}_i}({\bs x})\d{\bs x}\isdef
\int_{\Omega}f({\bs x})\delta_{{\bs x}_i}(\d{\bs x})
=f({\bs x}_i),
\]
where $f\in C(\Omega)$ is a continuous function.

Next, we define the space
\(V\isdef\spn\{\delta_{{\bs x}_1},\ldots,\delta_{{\bs x}_N}\}\)
as the \(N\)-dimensional vector space
of all discrete and finite signed measures supported at the
points in \(X\).
An inner product on \(V\) is defined by
\[
\langle u,v\rangle_V\isdef\sum_{i=1}^N u_iv_i,\quad\text{where }
u=\sum_{i=1}^Nu_i\delta_{{\bs x}_i},\ v=\sum_{i=1}^Nv_i\delta_{{\bs x}_i}.
\]
Indeed, the space \(V\) is isometrically isomorphic to \(\Rbb^N\)
endowed with the canonical inner product.
Similar to the
idea of a multiresolution analysis in the construction of
wavelets, we introduce the spaces \(V_j\isdef\spn{\bs \Phi_j}\), where
\[
{\bs\Phi_j}\isdef\{\varphi_{j,k}:k\in\Delta_j\}.
\]
Here, $\Delta_j$ denotes a suitable 
index set with cardinality $|\Delta_j|=\dim V_j$ and
\(j\in\Nbb\) is referred to as \emph{level}. 
Moreover, each basis element \(\varphi_{j,k}\) is a linear
combination of Dirac measures
such that
\[
\langle \varphi_{j,k},\varphi_{j,k'}\rangle_V=0\quad\text{for }k\neq k'.
\]

For the sake of notational convenience, we shall identify bases
by row vectors,
such that, for ${\bs v}_j 
= [v_{j,k}]_{k\in\Delta_j}$, the corresponding measure 
can simply be written as a dot product according to
\[
v_j = \mathbf\Phi_j{\bs v}_j=\sum_{k\in\Delta_j} v_{j,k}\varphi_{j,k}.
\]

Rather than using the multiresolution 
analysis corresponding to the hierarchy
\[
V_0\subset V_1\subset\cdots\subset V,
\]
the idea of samplets is
to keep track of the increment of information 
between two consecutive levels $j$ and $j+1$. Since we have
$V_{j}\subset V_{j+1}$, we may decompose 
\begin{equation}\label{eq:decomposition}
V_{j+1} = V_j\overset{\perp}{\oplus} S_j
\end{equation}
by using the \emph{detail space} $S_j$. Of practical interest 
is the particular choice of the basis of the detail space $S_j$ in $V_{j+1}$. 
This basis is assumed to be orthonormal as well and will be denoted by
\[
  {\bs\Sigma}_j = \{\sigma_{j,k}:k\in\nabla_j\isdef\Delta_{j+1}
  \setminus \Delta_j\}.
\]
Recursively applying the decomposition \eqref{eq:decomposition},
we see that the set
\[
\mathbf\Sigma_J = {\bs\Phi}_0\cup \bigcup_{j=0}^{J-1}{\bs\Sigma}_j
\]
forms a basis of \(V_J\isdef V\), which we call a \emph{samplet basis}. 
In view of data compression, an essential ingredient is the vanishing moment
condition, meaning that
\begin{equation}\label{eq:vanishingMoments}
 (p,\sigma_{j,k})_\Omega
 = 0\quad \text{for all}\ p\in\Pcal_q(\Omega),
\end{equation}
where \(\Pcal_q(\Omega)\) denotes the space of all polynomials
with total degree at most \(q\).
We say then that the samplets have $q+1$ \emph{vanishing 
moments}.

\begin{remark}
In case of uniformly distributed points, we can obtain bases
which satisfy
\[
\diam(\supp\varphi_{j,k})\isdef
\diam(\{{\bs x}_{i_1},\ldots,{\bs x}_{i_p}\})\sim 2^{-j/d}
\]
and, likewise,
\begin{equation}\label{eq:localized}
\diam(\supp\sigma_{j,k})\sim 2^{-j/d}.
\end{equation}
These properties are favorable with regard to the 
compression of data and kernel matrices. 
However, we stress that this is not a requirement in our construction.
\end{remark}

\begin{remark}
The concept of samplets has a very natural interpretation 
in the context of reproducing kernel Hilbert spaces, compare
\cite{Aronszajn50}. If \((\Hcal,\langle\cdot,\cdot\rangle_{\Hcal})\)
is a reproducing kernel Hilbert space with reproducing kernel
\(\kernel\), then there holds
\((f,\delta_{{\bs x}_i})_\Omega
=\langle \kernel({\bs x}_i,\cdot),f\rangle_{\Hcal}\). Hence,
the samplet
\(\sigma_{j,k}=\sum_{\ell=1}^p\beta_\ell\delta_{{\bs x}_{i_\ell}}\)
can directly be identified with the function
\[
\hat{\sigma}_{j,k}\isdef
\sum_{\ell=1}^p\beta_\ell \kernel({\bs x}_{i_\ell},\cdot)\in\mathcal{H}.
\]
In particular, it holds
\[
\langle\hat{\sigma}_{j,k},h\rangle_\Hcal=0
\]
for any \(h\in\Hcal\) which satisfies
\(h|_{\supp\sigma_{j,k}}\in\Pcal_q(\supp\sigma_{j,k})\).
\end{remark}

\section{Construction of samplets}\label{sct:construction}
\subsection{Cluster tree}
In order to construct samplets with the desired properties, 
especially vanishing moments, cf.\ \eqref{eq:vanishingMoments},
we shall transfer the wavelet construction of Tausch and 
White from \cite{TW03} into our setting. The first step is to 
construct a hierarchy subspaces of signed measures.
To this end, we perform a hierarchical 
clustering on the set \(X\).

\begin{definition}\label{def:cluster-tree}
Let $\mathcal{T}=(P,E)$ be a tree with vertices $P$ and edges $E$.
We define its set of leaves as
\[
\mathcal{L}(\mathcal{T})\isdef\{\nu\in P\colon\nu~\text{has no sons}\}.
\]
The tree $\mathcal{T}$ is a \emph{cluster tree} for
the set $X=\{{\bs x}_1,\ldots,{\bs x}_N\}$, iff
the set $X$ is the \emph{root} of $\mathcal{T}$ and
all $\nu\in P\setminus\mathcal{L}(\mathcal{T})$
are disjoint unions of their sons.

The \emph{level} \(j_\nu\) of $\nu\in\mathcal{T}$ is its distance from the root,
i.e.\ the number of son relations that are required for traveling from
$X$ to $\nu$. The \emph{depth} \(J\) of \(\Tcal\) is the maximum level
of all clusters. We define the set of clusters
on level $j$ as
\[
\mathcal{T}_j\isdef\{\nu\in\mathcal{T}\colon \nu~\text{has level}~j\}.
\]
Finally, the \emph{bounding box} $B_{\nu}$ of \(\nu\)
is defined as the smallest axis-parallel cuboid that 
contains all its points.
\end{definition}

There exist several possibilities for the choice of a
cluster tree for the set \(X\). However, within this article,
we will exclusively consider binary trees and remark that it is of course
 possible to consider other options, such as
\(2^d\)-trees, with the obvious modifications.
Definition~\ref{def:cluster-tree} provides a hierarchical cluster
structure on the set \(X\). Even so, it does not provide guarantees 
for the cardinalities of the clusters.
Therefore, we introduce the concept
of a balanced binary tree.

\begin{definition}
Let $\Tcal$ be a cluster tree on $X$ with depth $J$. 
$\Tcal$ is called a \emph{balanced binary tree}, if all 
clusters $\nu$ satisfy the following conditions:
\begin{enumerate}
\item
The cluster $\nu$ has exactly two sons
if $j_{\nu} < J$. It has no sons if $j_{\nu} = J$.
\item
It holds $|\nu|\sim 2^{J-j_{\nu}}$.
\end{enumerate}
\end{definition}

A balanced binary tree can be constructed by \emph{cardinality 
balanced clustering}. This means that the root cluster 
is split into two son clusters of identical (or similar)
cardinality. This process is repeated recursively for the 
resulting son clusters until their cardinality falls below a 
certain threshold.
For the subdivision, the cluster's bounding box
is split along its longest edge such that the 
resulting two boxes both contain an equal number of points.
Thus, as the cluster cardinality halves with each level, 
we obtain $\mathcal{O}(\log N)$ levels in total. 
The total cost for constructing the cluster tree
is therefore $\mathcal{O}(N \log N)$. Finally, we remark that a
balanced tree is only required to guarantee the cost bounds
for the presented algorithms. The error and compression estimates
we shall present later on are robust in the sense that they
are formulated directly in terms of the actual cluster sizes
rather than the introduced cluster level.

\subsection{Multiscale hierarchy}
Having a cluster tree at hand, we 
shall now construct a samplet basis on the resulting 
hierarchical structure. We begin by introducing a \emph{two-scale} 
transform between basis elements on a cluster $\nu$ of level $j$. 
To this end, we create \emph{scaling functions} $\mathbf{\Phi}_{j}^{\nu} 
= \{ \varphi_{j,k}^{\nu} \}$ and \emph{samplets} $\mathbf{\Sigma}_{j}^{\nu} 
= \{ \sigma_{j,k}^{\nu} \}$ as linear combinations of the scaling 
functions $\mathbf{\Phi}_{j+1}^{\nu}$ of $\nu$'s son clusters. 
This results in the \emph{refinement relation}
\begin{equation}\label{eq:refinementRelation}
   [ \mathbf{\Phi}_{j}^{\nu}, \mathbf{\Sigma}_{j}^{\nu} ] 
 \isdef 
 \mathbf{\Phi}_{j+1}^{\nu}
 {\bs Q}_j^{\nu}=
 \mathbf{\Phi}_{j+1}^{\nu}
 \big[ {\bs Q}_{j,\Phi}^{\nu},{\bs Q}_{j,\Sigma}^{\nu}\big].
\end{equation}

In order to provide both, vanishing moments and orthonormality,
the transformation \({\bs Q}_{j}^{\nu}\) has to be
appropriately constructed. For this purpose, we consider an orthogonal
decomposition of the \emph{moment matrix}
\[
  {\bs M}_{j+1}^{\nu}\isdef
  \begin{bmatrix}({\bs x}^{\bs 0},\varphi_{j+1,1})_\Omega&\cdots&
  ({\bs x}^{\bs 0},\varphi_{j+1,|\nu|})_\Omega\\
  \vdots & & \vdots\\
  ({\bs x}^{\bs\alpha},\varphi_{j+1,1})_\Omega&\cdots&
  ({\bs x}^{\bs\alpha},\varphi_{j+1,|\nu|})_\Omega
  \end{bmatrix}=
  [({\bs x}^{\bs\alpha},\mathbf{\Phi}_{j+1}^{\nu})_\Omega]_{|\bs\alpha|\le q}
\in\Rbb^{m_q\times|\nu|},
\]
where
\begin{equation}\label{eq:mq}
m_q\isdef\sum_{\ell=0}^q{\ell+d-1\choose d-1}={q+d\choose d}\leq(q+1)^d
\end{equation}
denotes the dimension of \(\Pcal_q(\Omega)\).

In the original construction by
Tausch and White, the matrix \({\bs Q}_{j}^{\nu}\) is obtained
from a singular value decomposition of \({\bs M}_{j+1}^{\nu}\).
For the construction of samplets, we follow the idea
form \cite{AHK14} and rather
employ the QR decomposition, which has the advantage of generating
samplets with an increasing number of vanishing moments.
It holds
\begin{equation}\label{eq:QR} 
  ({\bs M}_{j+1}^{\nu})^\intercal  = {\bs Q}_j^\nu{\bs R}
  \defis\big[{\bs Q}_{j,\Phi}^{\nu} ,
  {\bs Q}_{j,\Sigma}^{\nu}\big]{\bs R}
 \end{equation}
Consequently, the moment matrix 
for the cluster's own scaling functions and samplets is then
given by
\begin{equation}\label{eq:vanishingMomentsQR}
  \begin{aligned}
  \big[{\bs M}_{j,\Phi}^{\nu}, {\bs M}_{j,\Sigma}^{\nu}\big]
  &= \left[({\bs x}^{\bs\alpha},[\mathbf{\Phi}_{j}^{\nu},
  	\mathbf{\Sigma}_{j}^{\nu}])_\Omega\right]_{|\bs\alpha|\le q}
  = \left[({\bs x}^{\bs\alpha},\mathbf{\Phi}_{j+1}^{\nu}[{\bs Q}_{j,\Phi}^{\nu}
  , {\bs Q}_{j,\Sigma}^{\nu}])_\Omega
  	\right]_{|\bs\alpha|\le q} \\
  &= {\bs M}_{j+1}^{\nu} [{\bs Q}_{j,\Phi}^{\nu} , {\bs Q}_{j,\Sigma}^{\nu} ]
  = {\bs R}^\intercal.
  \end{aligned}
\end{equation}
As ${\bs R}^\intercal$ is a lower triangular matrix, the first $k-1$ 
entries in its $k$-th column are zero. This corresponds to 
$k-1$ vanishing moments for the $k$-th function generated 
by the transformation
${\bs Q}_{j}^{\nu}=[{\bs Q}_{j,\Phi}^{\nu} , {\bs Q}_{j,\Sigma}^{\nu} ]$. 
By defining the first $m_{q}$ functions as scaling functions and 
the remaining ones as samplets, we obtain samplets with vanishing 
moments at least up to order $q+1$. By increasing
the polynomial degree to \(\hat{q}> q\) at the leaf clusters
such that \(m_{\hat{q}}\geq 2m_q\), we can even construct 
samplets with an increased number of vanishing moments up to order \(\hat{q}+1\)
without any additional cost. 

\begin{remark}
We remark that the samplet construction using vanishing moments
is inspired by the classical wavelet theory. However, it is easily
possible to adapt the construction to other primitives of interest.
\end{remark}

\begin{remark}
\label{remark:introCQ}
Each cluster has at most a constant number of scaling 
functions and samplets: For a particular cluster $\nu$, their number 
is identical to the cardinality of $\mathbf{\Phi}_{j+1}^{\nu}$. For leaf 
clusters, this number is bounded by the leaf size. 
For non-leaf clusters, it is bounded by the number of scaling functions 
provided from all its son clusters. As there are at most two 
son clusters with a maximum of $m_q$ scaling functions each, 
we obtain the bound $2 m_q$ for non-leaf clusters. Note that, 
if $\mathbf{\Phi}_{j+1}^{\nu}$ has at most $m_q$ elements, a 
cluster will not provide any samplets at all and all functions 
will be considered as scaling functions.
\end{remark}

For leaf clusters, we define the scaling functions by
the Dirac measures supported at the points \({\bs x}_i\), i.e.\
$\mathbf{\Phi}_J^{\nu}\isdef\{ \delta_{{\bs x}_i} : {\bs x}_i\in\nu \}$.
The scaling functions of all clusters on a specific level $j$ 
then generate the spaces
\begin{equation}\label{eq:Vspaces}
	V_{j}\isdef \spn\{ \varphi_{j,k}^{\nu} : k\in \Delta_j^\nu,\ \nu \in\Tcal_{j} \},
\end{equation}
while the samplets span the detail spaces
\begin{equation}\label{eq:Wspaces}
	S_{j}\isdef
	\spn\{ \sigma_{j,k}^{\nu} : k\in \nabla_j^\nu,\ 
	\nu \in \Tcal_{j} \} =
	V_{j+1}\overset{\perp}{\ominus} V_j.
\end{equation}
Combining the scaling functions of the root cluster with all 
clusters' samplets gives rise to the samplet basis
\begin{equation}\label{eq:Wbasis}
  \mathbf{\Sigma}_{N}\isdef\mathbf{\Phi}_{0}^{X} 
  	\cup \bigcup_{\nu \in T} \mathbf{\Sigma}_{j}^{\nu}.
\end{equation}

Writing $\mathbf{\Sigma}_{N}
= \{ \sigma_{k} : 1 \leq k \leq N \}$, where 
$\sigma_{k}$ is either a samplet or a scaling function 
at the root cluster, we can establish a unique indexing of 
all the signed measures comprising the samplet 
basis. The indexing induces an order on the 
basis set $\mathbf{\Sigma}_{N}$, which we choose 
to be level-dependent: Samplets belonging to a particular 
cluster are grouped together, with those on finer levels 
having larger indices.

\begin{remark}\label{remark:waveletLeafSize}
We remark that the samplet basis on a balanced 
cluster tree can be computed in cost $\mathcal{O}(N)$, 
we refer to \cite{AHK14} for a proof of this statement.
\end{remark}

\subsection{Properties of the samplets}
By construction, samplets satisfy the following 
properties, which can directly be inferred from
the corresponding results in \cite{HKS05,TW03}.

\begin{theorem}\label{theo:waveletProperties}
The spaces $V_{j}$ defined in equation \eqref{eq:Vspaces} 
exhibit the desired multiscale hierarchy
\[
  V_0\subset V_1\subset\cdots\subset V_J = V,
\]
where the corresponding complement spaces $S_{j}$ from \eqref{eq:Wspaces} 
satisfy $V_{j+1}=V_j\overset{\perp}{\oplus} S_{j}$ for all $j=0,1,\ldots,
J-1$. The associated samplet basis $\mathbf{\Sigma}_{N}$ defined in 
\eqref{eq:Wbasis} forms an orthonormal basis of $V$. 
In particular, there holds:
\begin{enumerate}
\item[(\emph{i})] The number of all samplets on level $j$ behaves like $2^j$.
\item[(\emph{ii})] The samplets have $q+1$ vanishing moments.
\item[(\emph{iii})] 
Each samplet is supported in a specific cluster $\nu$. 
\end{enumerate}
\end{theorem}

\begin{remark}
In the situation of Theorem~\ref{theo:waveletProperties},
if the points in $X$ are even uniformly distributed, then the 
diameter of the cluster satisfies $\diam(\nu)\sim 
2^{-j_\nu/d}$ and it holds \eqref{eq:localized}. 
\end{remark}

\begin{remark}
Due to $S_j\subset V$ and $V_0\subset V$, 
we conclude that each samplet is a linear combination of the 
Dirac measures supported at the points in $X$. Especially, the 
related coefficient vectors ${\bs\omega}_{j,k}$ in
\begin{equation}\label{eq:coefficientVectorsOfWavelets}
  \sigma_{j,k} = \sum_{i=1}^{N}
  \omega_{j,k,i} \delta_{{\bs x}_i} \quad 
	\text{and} \quad \varphi_{0,k} = \sum_{i=1}^{N}
	\omega_{0,k,i} \delta_{{\bs x}_i}
\end{equation}
are pairwise orthonormal with respect to the inner
product on \(\Rbb^N\).
\end{remark}

Later on, the following bound on the samplets' 
coefficients $\|\cdot\|_1$-norm will
be essential:

\begin{lemma}\label{lemma:waveletL1Norm}
The coefficient vector ${\bs\omega}_{j,k}=\big[\omega_{j,k,i}\big]_i$ of
the samplet $\sigma_{j,k}$ on the cluster $\nu$ fulfills
\begin{equation}\label{eq:ell1-norm}
  \|{\bs\omega}_{j,k}\|_{1}\le\sqrt{|\nu|}.
\end{equation}
The same holds for the scaling functions $\varphi_{j,k}$.
\end{lemma}

\begin{proof}
It holds $\|{\bs\omega}_{j,k}\|_{\ell^2}=1$. Hence,
the assertion follows immediately from the Cauchy-Schwarz
inequality
\[
\|{\bs\omega}_{j,k}\|_{1}\le\sqrt{|\nu|}\|{\bs\omega}_{j,k}\|_{2}
=\sqrt{|\nu|}.
\]
\end{proof}

The key for data compression and singularity detection 
is the following estimate which shows that the samplet 
coefficients decay with respect to the samplet's level 
provided that the data result from the evaluation of a smooth function. 
Therefore, in case of smooth data, the samplet 
coefficients are small and can be set to zero without 
compromising the accuracy. Vice versa, a large samplet 
coefficients reflects that the data are singular in the 
region of the samplet's support. 

\begin{lemma}\label{lemma:decay}
Let $f\in C^{q+1}(\Omega)$. Then, it holds for
a samplet $\sigma_{j,k}$ supported
on the cluster $\nu$ that
\begin{equation}\label{eq:decay}
 |(f,\sigma_{j,k})_\Omega|\le
  	\diam(\nu)^{q+1}\|f\|_{C^{q+1}(\Omega)}\|{\bs\omega}_{j,k}\|_{1}.
\end{equation}
\end{lemma}

\begin{proof}
For ${\bs x}_0\in\nu$, a Taylor expansion of $f$ yields
\[
f({\bs x}) = \sum_{|\bs\alpha|\le q}
\frac{\partial^{|\bs\alpha|}}{\partial{\bs x}^{\bs\alpha}}f({\bs x}_0)
	\frac{({\bs x}-{\bs x}_0)^{\bs\alpha}}{\bs\alpha!}
		+ R_{{\bs x}_0}({\bs x}).
\]
Herein, the remainder $R_{{\bs x}_0}({\bs x})$ reads
\begin{align*}
  R_{{\bs x}_0}({\bs x}) &= (q+1)\sum_{|\bs\alpha|=q+1}
  	\frac{({\bs x}-{\bs x}_0)^{\bs\alpha}}{\bs\alpha!}
	\int_0^1\frac{\partial^{q+1}}{\partial{\bs x}^{\bs\alpha}}
	f\big({\bs x}_0+s({\bs x}-{\bs x}_0)\big)(1-s)^q\d s.
\end{align*}
In view of the vanishing moments, we conclude
\begin{align*}
  |(f,\sigma_{j,k})_\Omega|
  	&= |(R_{{\bs x}_0},\sigma_{j,k})_\Omega|
	\le\sum_{|\bs\alpha|=q+1}
	\frac{\|{\bs x}-{\bs x}_0\|_2^{|\bs\alpha|}}{\bs\alpha!}
	\max_{{\bs x}\in\nu}\bigg|\frac{\partial^{q+1}}
		{\partial{\bs x}^{\bs\alpha}}f(\bs x)\bigg|
		\|{\bs\omega}_{j,k}\|_{1}\\
	&\le\diam(\nu)^{q+1}\|f\|_{C^{q+1}(\Omega)}\|{\bs\omega}_{j,k}\|_{1}.
\end{align*}
Here, we used the estimate
\[
 \sum_{|\bs\alpha|=q+1}\frac{2^{-(q+1)}}{\bs\alpha!}\le 1,
\]
which is obtained by choosing \({\bs x}_0\) as the
cluster's midpoint.
\end{proof}

\section{Discrete samplet transform}\label{sec:FST}
In order to transform between the samplet basis 
and the basis of Dirac measures, we introduce
the \emph{discrete samplet transform} and its inverse.
To this end, we assume that the data 
\(({\bs x}_1,y_1),\ldots,({\bs x}_N,y_N)\)
result from the evaluation of some (unknown) function
\(f\colon\Omega\to\Rbb\),
i.e.\ 
\[y_i=f_i^{\Delta}=(f,\delta_{{\bs x}_i})_\Omega.
\]
Hence, we may represent the function \(f\) on \(X\)
according to
\[f = \sum_{i = 1}^{N} f_i^{\Delta} \delta_{{\bs x}_i}.
\]
Our goal is now to compute the representation
\[f = 
\sum_{i = 1}^{N} f_{k}^{\Sigma} \sigma_{k}
\]
with respect to the samplet basis.
For 
sake of a simpler notation, let
${\bs f}^{\Delta}\isdef [f_i^{\Delta}]_{i=1}^N$ 
and ${\bs f}^{\Sigma}\isdef [f_i^\Sigma]_{i=1}^N$ denote
the associated coefficient vectors.

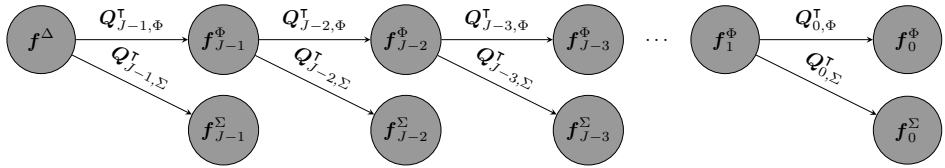
\begin{figure}[htb]
\begin{center}
\scalebox{0.75}{
\begin{tikzpicture}[x=0.4cm,y=0.4cm]
\tikzstyle{every node}=[circle,draw=black,fill=shadecolor,
minimum size=1.2cm]%
\tikzstyle{ptr}=[draw=none,fill=none,above]%
\node at (0,5) (1) {${\bs f}^{\Delta}$};
\node at (8,5) (2) {${\bs f}_{J-1}^{\Phi}$};
\node at (8,1) (3) {${\bs f}_{J-1}^{\Sigma}$};
\node at (16,5) (4) {${\bs f}_{J-2}^{\Phi}$};
\node at (16,1) (5) {${\bs f}_{J-2}^{\Sigma}$};
\node at (24,5) (6) {${\bs f}_{J-3}^{\Phi}$};
\node at (24,1) (7) {${\bs f}_{J-3}^{\Sigma}$};
\node at (30,5) (8) {${\bs f}_{1}^{\Phi}$};
\node at (38,5) (9) {${\bs f}_{0}^{\Phi}$};
\node at (38,1) (10) {${\bs f}_{0}^{\Sigma}$};
\tikzstyle{forward}=[draw,-stealth]%
\tikzstyle{every node}=[style=ptr]
\draw
(1) edge[forward] node[above,sloped]{${\bs Q}_{J-1,\Phi}^\intercal$} (2)
(1) edge[forward] node[above,sloped]{${\bs Q}_{J-1,\Sigma}^\intercal$}%
 (3)
(2) edge[forward] node[above,sloped]{${\bs Q}_{J-2,\Phi}^\intercal$} (4)
(2) edge[forward] node[above,sloped]{${\bs Q}_{J-2,\Sigma}^\intercal$}%
 (5)
(4) edge[forward] node[above,sloped]{${\bs Q}_{J-3,\Phi}^\intercal$} (6)
(4) edge[forward] node[above,sloped]{${\bs Q}_{J-3,\Sigma}^\intercal$}%
 (7)
(8) edge[forward] node[above,sloped]{${\bs Q}_{0,\Phi}^\intercal$} (9)
(8) edge[forward] node[above,sloped]{${\bs Q}_{0,\Sigma}^\intercal$}%
  (10);
\tikzstyle{every node}=[style=ptr]%
\tikzstyle{ptr}=[draw=none,fill=none]%
\node at (27,5) (16) {$\hdots$};
\end{tikzpicture}}
\caption{\label{fig:haar}Visualization of the discrete samplet transform.}
\end{center}
\end{figure}

The discrete samplet transform is based on 
recursively applying the refinement relation 
\eqref{eq:refinementRelation} to the point evaluations
\begin{equation}\label{eq:refinementRelationInnerProducts}
(f, [ \mathbf{\Phi}_{j}^{\nu}, \mathbf{\Sigma}_{j}^{\nu} ])_\Omega
=(f, \mathbf{\Phi}_{j+1}^{\nu} [{\bs Q}_{j,\Phi}^{\nu} ,{\bs Q}_{j,\Sigma}^{\nu} ])_\Omega \\
=(f, \mathbf{\Phi}_{j+1}^{\nu})_\Omega [{\bs Q}_{j,\Phi}^{\nu} , {\bs Q}_{j,\Sigma}^{\nu} ].
\end{equation}
On the finest level, the entries of the vector
$(f, \mathbf{\Phi}_{J}^{\nu})_\Omega$ 
are exactly those of ${\bs f}^{\Delta}$. Recursively 
applying equation \eqref{eq:refinementRelationInnerProducts} therefore
yields all the coefficients $(f, \mathbf{\Sigma}_{j}^{\nu})_\Omega$,
including $(f, \mathbf{\Phi}_{0}^{X})_\Omega$,
required for the representation of $f$ in the samplet basis,
see Figure~\ref{fig:haar} for a visualization of the resulting 
fish bone scheme. The
complete procedure is 
 formulated in Algorithm~\ref{algo:DWT}.

\begin{center}
\scalebox{0.8}{
\begin{algorithm}[H]
\caption{Discrete samplet transform}
\label{algo:DWT}	
\KwData{Data ${\bs f}^\Delta$,
cluster tree $\Tcal$ and transformations
$[{\bs Q}_{j,\Phi}^{\nu},{\bs Q}_{j,\Sigma}^{\nu}]$.}
\KwResult{Coefficients ${\bs f}^{\Sigma}$ 
stored as
$[(f,\mathbf{\Phi}_{0}^{X})_\Omega]^\intercal$ and 
$[(f,\mathbf{\Sigma}_{j}^{\nu})_\Omega]^\intercal$.}	
\Begin{
store $[(f,\mathbf{\Phi}_{0}^{X})_\Omega]^\intercal\isdef$ 
\FuncSty{transformForCluster}($X$)
}
\end{algorithm}}
\end{center}
\begin{center}
\scalebox{0.8}{
\begin{function}[H]
\caption{transformForCluster($\nu$)}
\Begin{
\uIf{$\nu=\{{\bs x}_{i_{1}}, \dots,{\bs x}_{i_{|\nu|}}\}$
	is a leaf of \(\Tcal\)}{
	set ${\bs f}_{j+1}^{\nu}\isdef
	\big[f_{i_{k}}^\Delta\big]_{k=1}^{|\nu|}$
	}
\Else{
\For{all sons $\nu'$ of $\nu$}{
execute $\transformForCluster(\nu')$\\
append the result to ${\bs f}_{j+1}^{\nu}$
}
}
set $[(f,\mathbf{\Sigma}_{j}^{\nu})_\Omega]^\intercal
\isdef({\bs Q}_{j,\Sigma}^{\nu})^\intercal {\bs f}_{j+1}^{\nu}$

\Return{$({\bs Q}_{j,\Phi}^{\nu})^\intercal{\bs f}_{j+1}^{\nu}$}
}
\end{function}}
\end{center}
\begin{remark}
Algorithm \ref{algo:DWT} is based on the transposed version of 
\eqref{eq:refinementRelationInnerProducts} to preserve
the column vector structure of ${\bs f}^\Delta$ and ${\bs f}^{\Sigma}$.
\end{remark}

The inverse transformation is obtained by reversing 
the steps of the discrete samplet transform:
For each cluster, we compute
\[
(f, \mathbf{\Phi}_{j+1}^{\nu})_\Omega
= (f, [ \mathbf{\Phi}_{j}^{\nu}, \mathbf{\Sigma}_{j}^{\nu} ] 
)_\Omega[{\bs Q}_{j,\Phi}^{\nu} ,{\bs Q}_{j,\Sigma}^{\nu} ]^\intercal
\] 
to either obtain the 
coefficients of the
son clusters' scaling functions
or, for leaf clusters, the coefficients ${\bs f}^{\Delta}$. 
The procedure is summarized in Algorithm~\ref{algo:iDWT}.

\begin{center}
\scalebox{0.8}{
\begin{algorithm}[H]
\caption{Inverse samplet transform}
\label{algo:iDWT}
\KwData{Coefficients ${\bs f}^\Sigma$,
cluster tree $\Tcal$ and transformations
$[{\bs Q}_{j,\Phi}^{\nu},{\bs Q}_{j,\Sigma}^{\nu}]$.}
\KwResult{Coefficients ${\bs f}^{\Delta}$
stored as
$[(f,\mathbf{\Phi}_{j}^{\nu})_\Omega]^\intercal$.}
\Begin{
\FuncSty{inverseTransformForCluster}($X$,
$[(f,\mathbf{\Phi}_{0}^{X})_\Omega]^\intercal$)
}
\end{algorithm}}
\end{center}
\begin{center}
\scalebox{0.8}{
\begin{function}[H]
\caption{inverseTransformForCluster($\nu$,
\unexpanded{$[(f,{\bs\Phi}_{j}^\nu)_\Omega]^\intercal$})}
\Begin{
$[(f,{\bs\Phi}_{j+1}^\nu)_\Omega]^\intercal
\isdef [{\bs Q}_{j,\Phi}^{\nu} ,{\bs Q}_{j,\Sigma}^{\nu} ]
\begin{bmatrix}
[(f,{\bs\Phi}_{j}^\nu)_\Omega]^\intercal\\
[(f,{\bs\Sigma}_{j}^\nu)_\Omega]^\intercal
\end{bmatrix}$

\uIf{$\nu=\{{\bs x}_{i_{1}}, \dots,{\bs x}_{i_{|\nu|}}\}$
is a leaf of \(\Tcal\)}{set $\big[f_{i_{k}}^\Delta\big]_{k=1}^{|\nu|}
\isdef[(f,{\bs\Phi}_{j_\nu+1}^\nu)_\Omega]^\intercal$
}
\Else{
\For{all sons $\nu'$ of $\nu$}{
assign the part of $[(f,{\bs\Phi}_{j+1}^\nu)_\Omega]^\intercal$
belonging to \(\nu'\) to $[(f,{\bs\Phi}_{j'}^{\nu'})_\Omega]^\intercal$
		        
execute \FuncSty{inverseTransformForCluster}($\nu'$,
$[(f,{\bs\Phi}_{j'}^{\nu'})_\Omega]^\intercal$)			}
}
}
\end{function}}
\end{center}

The discrete samplet transform and its inverse 
can be performed in linear cost. This 
result is well known in case of wavelets and was
crucial for their rapid development.

\begin{theorem}
The runtime of the discrete samplet transform and the inverse
samplet transform are \(\mathcal{O}(N)\), each.
\end{theorem}

\begin{proof}
As the samplet construction follows the construction
of Tausch and White, we refer to \cite{TW03} for the
details of the proof.
\end{proof}

\section{Numerical results I}\label{sec:Num1}
To demonstrate the efficacy of the samplet analysis,
we compress different sample data in one, two and three
spatial dimensions. For each example, we use samplets 
with \(q+1=3\) vanishing moments.

\subsection*{One dimension}
We start with two one-dimensional
examples. On the one hand, we consider the function
\[
f(x)=\frac 3 2 e^{-40|x-\frac 1 4|}
+ 2e^{-40|x|}-e^{-40|x+\frac 1 2|},
\]
sampled at $8192$ uniformly distributed points on \([-1,1]\).
On the other hand, we consider a path of a Brownian motion 
sampled at the same points. The coefficients of the samplet 
transformed data are thresholded with \(10^{-i}\|{\bs f}^{\Sigma}\|_\infty\), 
\(i=1,2,3\), respectively.
The resulting compression ratios and the reconstructions 
can be found in Figure~\ref{fig:Expcomp} and Figure~\ref{fig:BMcomp}, 
respectively. One readily infers that in both cases high compression 
rates are achieved at high accuracy. In case of the Brownian motion,
the smoothing of the sample data can be realized by increasing the
compression rate, corresponding to throwing away more and 
more detail information. Indeed, due to the orthonormality of the samplet
basis, this procedure amounts to a least squares fit of the data.

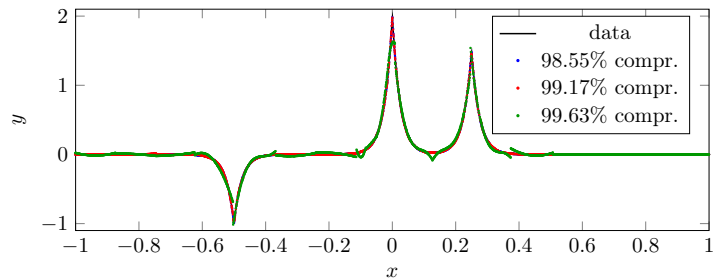
\begin{figure}[htb]
\begin{center}
\scalebox{0.8}{
\begin{tikzpicture}
\begin{axis}[width=0.8\textwidth, height=0.35\textwidth, xmin = -1, xmax=1,
 ymin=-1.1,
ymax=2.1, ylabel={$y$}, xlabel ={$x$},legend style={mark
 options={scale=2}},
legend pos = north east]
\addplot[line width=0.7pt,color=black]
table[each nth point=3,x index={0},y index = {1}]{%
./Results/ExpCompress1D.txt};
\addlegendentry{data};
\addplot[line width=0.7pt,color=blue, only marks, mark size=0.2pt]
table[each nth point=3,x index={0},y index = {5}]{%
./Results/ExpCompress1D.txt};
\addlegendentry{$98.55\%$ compr.};
\addplot[line width=0.7pt,color=red, only marks, mark size=0.2pt]
table[each nth point=3,x index={0},y index = {4}]{%
./Results/ExpCompress1D.txt};
\addlegendentry{$99.17\%$ compr.};
\addplot[line width=0.7pt,color=darkgreen, only marks, mark size=0.2pt]
table[each nth point=3,x index={0},y index = {3}]{%
./Results/ExpCompress1D.txt};
\addlegendentry{$99.63\%$ compr.};
\end{axis}
\end{tikzpicture}}
\end{center}
\caption{\label{fig:Expcomp}Sampled function approximated with
different compression ratios.}
\end{figure}

\begin{figure}[htb]
\begin{center}
\scalebox{0.8}{
\begin{tikzpicture}[ausschnitt/.style={black!80}]
\begin{axis}[width=0.8\textwidth, height=0.35\textwidth, xmin = -1, xmax=1,
 ymin=-1,
ymax=2.4,
ylabel={$y$}, xlabel ={$x$},legend style={mark options={scale=2}},
legend pos = south east]
    \draw[ausschnitt]
      (axis cs:-0.5,-0.5)coordinate(ul)--
      (axis cs:0.005,-0.5)coordinate(ur)--
      (axis cs:0.005,0.4)coordinate(or)--
      (axis cs:-0.5,0.4) -- cycle;
\addplot[line width=0.7pt,color=black]
table[each nth point=5,x index={0},y index = {1}]{%
./Results/BMCompress1D.txt};
\addlegendentry{data};
\addplot[line width=0.7pt,color=blue, only marks, mark size=0.2pt]
table[each nth point=6,x index={0},y index = {5}]{%
./Results/BMCompress1D.txt};
\addlegendentry{$92.69\%$ compr.};
\addplot[line width=0.7pt,color=red, only marks, mark size=0.2pt]
table[each nth point=6,x index={0},y index = {4}]{%
./Results/BMCompress1D.txt};
\addlegendentry{$99.24\%$ compr.};
\addplot[line width=0.7pt,color=darkgreen, only marks, mark size=0.2pt]
table[each nth point=6,x index={0},y index = {3}]{%
./Results/BMCompress1D.txt};
\addlegendentry{$99.88\%$ compr.};
\end{axis}
\begin{axis}[yshift=-.32\textwidth,xshift=0.12\textwidth,
width=0.4\textwidth, height=0.35\textwidth, xmin = -0.5,
xmax=0.005, ymin=-0.5,
ymax=0.4,axis line style=ausschnitt]
\addplot[line width=0.7pt,color=black]
table[each nth point=2,x index={0},y index = {1}]{%
./Results/BMCompress1D.txt};
\addplot[line width=0.7pt,color=blue, only marks, mark size=0.2pt]
table[each nth point=2,x index={0},y index = {5}]{%
./Results/BMCompress1D.txt};
\addplot[line width=0.7pt,color=red, only marks, mark size=0.2pt]
table[each nth point=2,x index={0},y index = {4}]{%
./Results/BMCompress1D.txt};
\addplot[line width=0.7pt,color=darkgreen, only marks, mark size=0.2pt]
table[each nth point=2,x index={0},y index = {3}]{%
./Results/BMCompress1D.txt};
\end{axis}
    \draw[ausschnitt]
      (current axis.north west)--(ul)
      (current axis.north east)--(ur);
\end{tikzpicture}}
\caption{\label{fig:BMcomp}Sampled Brownian motion approximated with
different compression ratios.}
\end{center}
\end{figure}
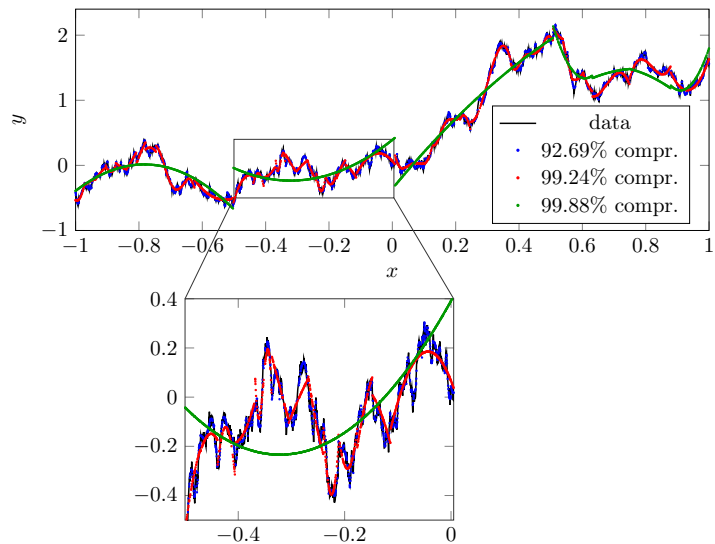

\subsection*{Two dimensions}
As a second application for samplets, we consider image compression.
To this end, we use a \(2000\times 2000\) pixel grayscale landscape 
image. The coefficients of the samplet transformed image are thresholded
with \(10^{-i}\|{\bs f}^{\Sigma}\|_\infty\), \(i=2,3,4\), respectively.
The corresponding
results and compression rates can be found in Figure~\ref{fig:compImage}.
A visualization of the samplet coefficients in case of the respective
low compression can be found in Figure~\ref{fig:coeffImage}.

\begin{figure}[htb]
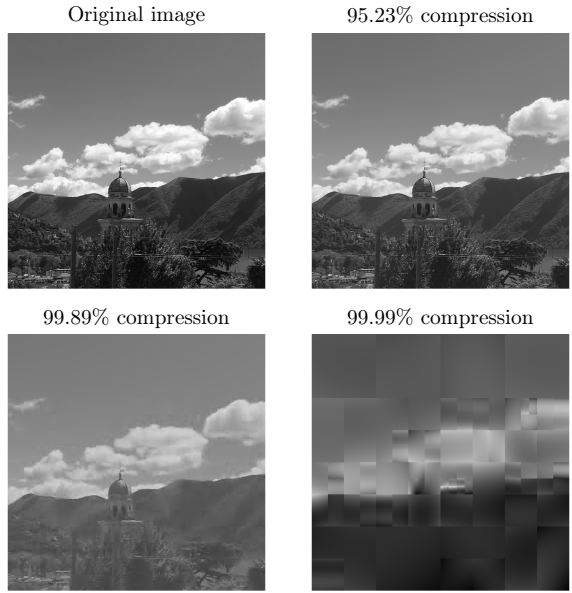

\begin{center}
\scalebox{0.8}{
\begin{tikzpicture}
\draw(0,0)node{\includegraphics[scale = 0.12,trim=65 47 65 24,clip]{%
./Results/OriginalLugano.png}};
\draw(0,2.4)node{Original image};
\draw(5,2.4)node{\(95.23\%\) compression};
\draw(5,0)node{\includegraphics[scale = 0.12,trim=65 47 65 24,clip]{%
./Results/CompressedLowLugano.png}};
\draw(0,-2.6)node{\(99.89\%\) compression};
\draw(0,-5)node{\includegraphics[scale = 0.12,trim=65 47 65 24,clip]{%
./Results/CompressedIntermedLugano.png}};
\draw(5,-2.6)node{\(99.99\%\) compression};
\draw(5,-5)node{\includegraphics[scale = 0.12,trim=65 47 65 24,clip]{%
./Results/CompressedHighLugano.png}};
\end{tikzpicture}}
\caption{\label{fig:compImage}Different compression rates of the
test image.}
\end{center}
\end{figure}

\begin{figure}[htb]
\begin{center}
\begin{tikzpicture}
\draw(0,0)node{\includegraphics[scale = 0.1,trim=1000 195 1000 195,clip]{%
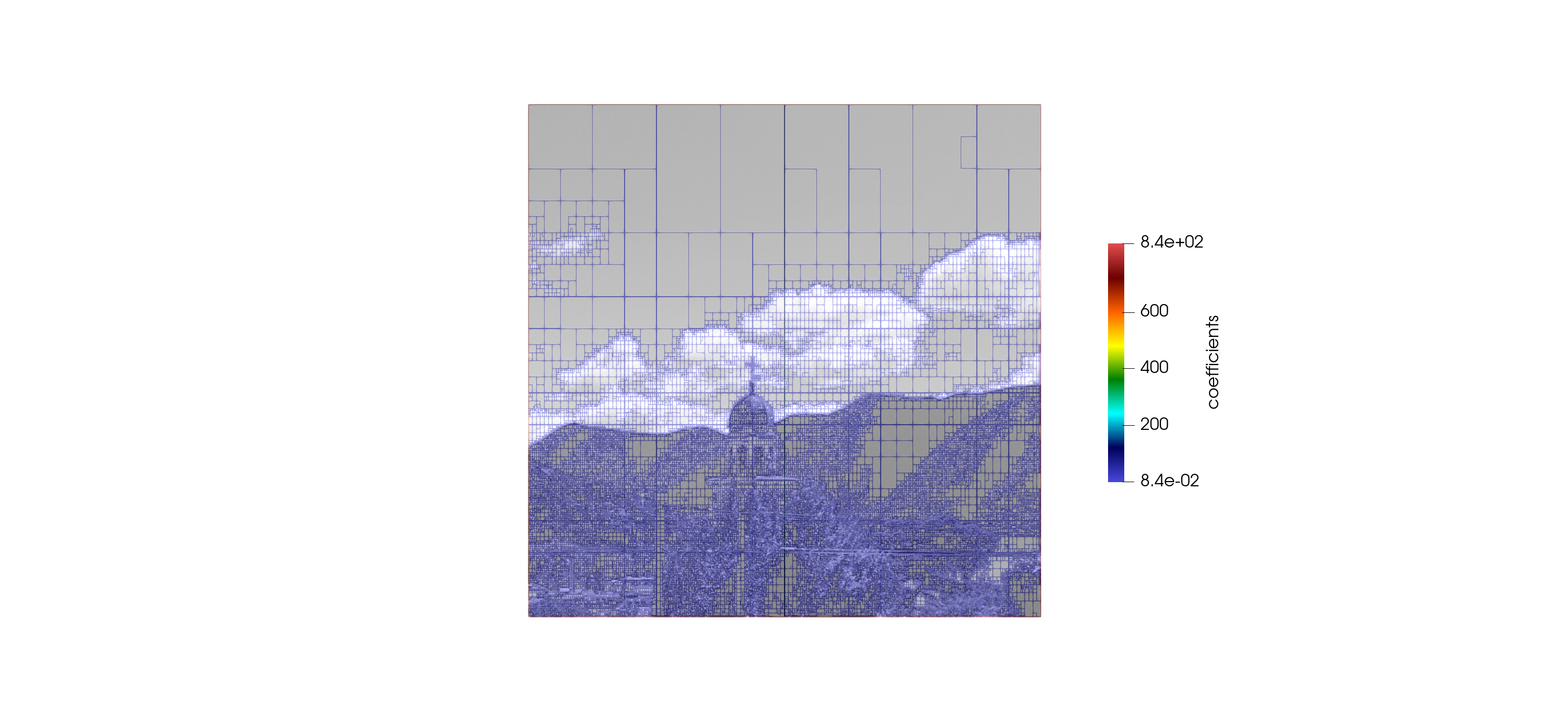}};
\draw(3,0)node{\includegraphics[scale = 0.16,trim=2100 400 460 400,clip]{%
./Results/LuganoCoeffs.png}};
\end{tikzpicture}
\caption{\label{fig:coeffImage}Visualization of the samplet coefficients
for the test image.}
\end{center}
\end{figure}

\subsection*{Three dimensions}
Finally, we show a result in three dimensions.
Here, the points are given by a uniform subsample of
a triangulation of the Stanford bunny. We consider data on the
Stanford bunny generated by the function
\[
f({\bs x})=e^{-20\|{\bs x}-{\bs p}_0\|_2}
+e^{-20\|{\bs x}-{\bs p}_1\|_2},
\]
where the points \({\bs p}_0\) and \({\bs p}_1\) are located at the tips
of the bunny's ears. Moreover, the geometry has been rescaled to a 
diameter of 2. The plot on the left-hand side of
Figure~\ref{fig:coeffStanford} 
visualizes the sample data, while the plot on the right-hand side 
shows the dominant coefficients in case of a threshold parameter 
of \(10^{-2}\|{\bs f}^{\Sigma}\|_\infty\).

\begin{figure}[htb]
\begin{center}
\begin{tikzpicture}
\draw(0,0)node{\includegraphics[%
scale = 0.1,trim=1040 230 1050 280,clip]{%
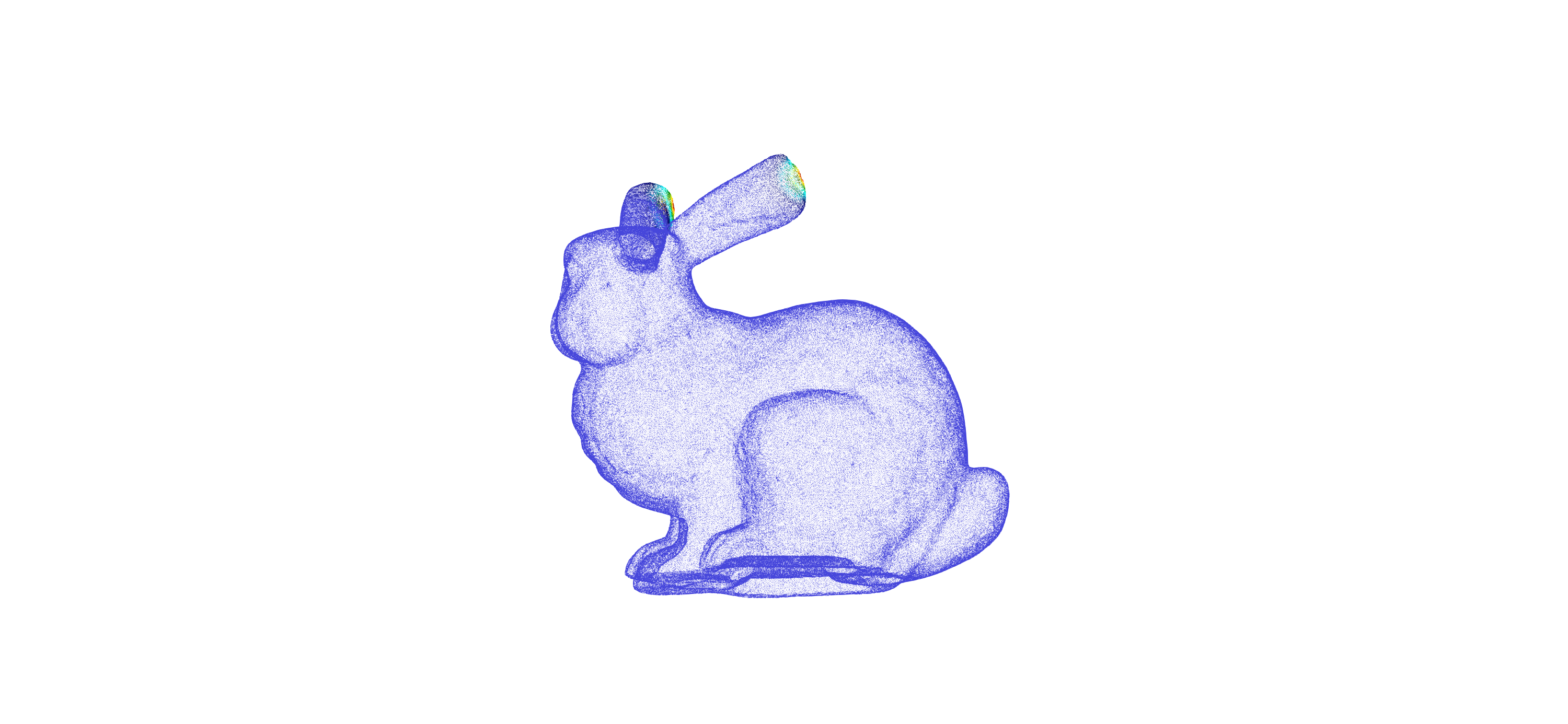}};
\draw(4,0)node{\includegraphics[%
scale = 0.1,trim=1000 200 1000 200,clip]{%
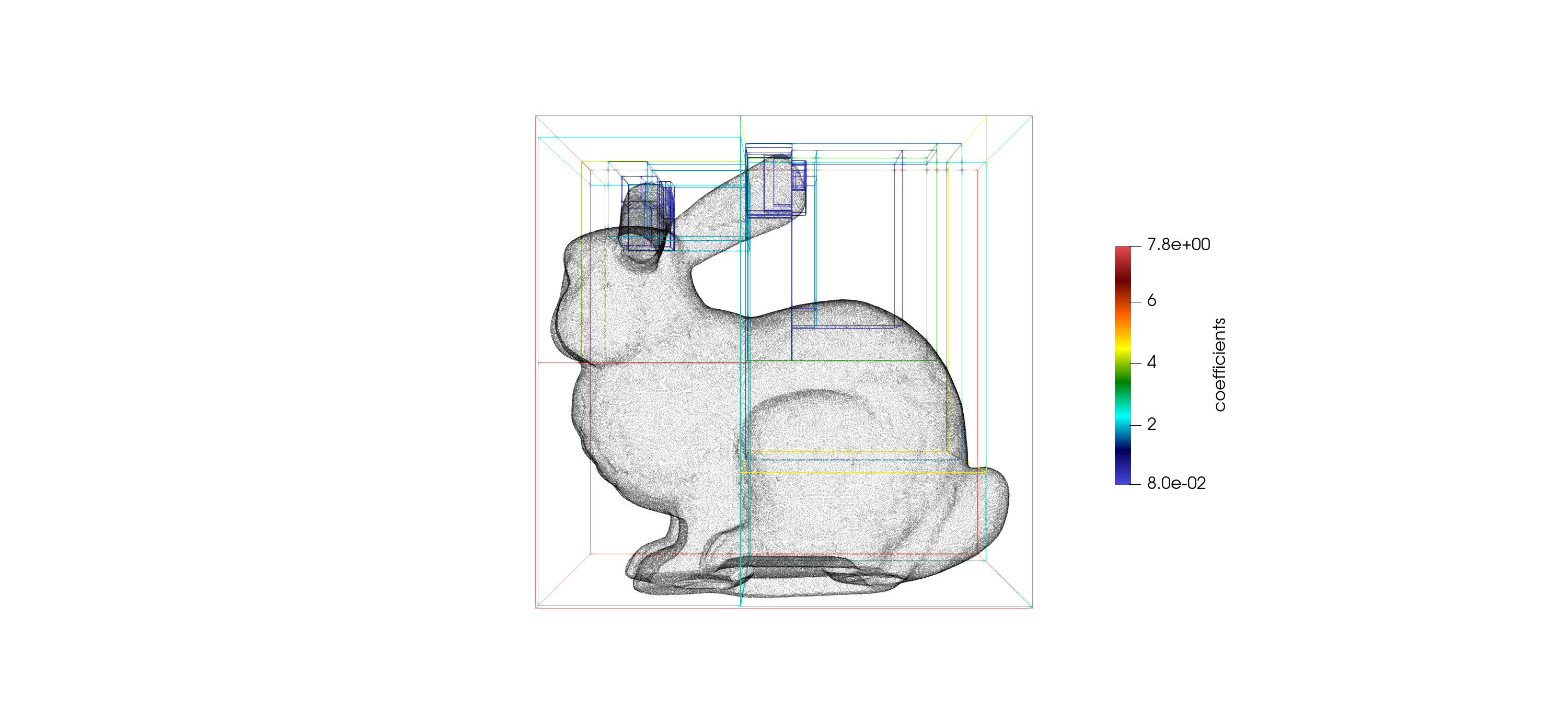}};
\draw(7,0)node{\includegraphics[%
scale = 0.16,trim=2130 400 600 400,clip]{%
./Results/StanfordBunny1e-2Coeff.png}};
\end{tikzpicture}
\caption{\label{fig:coeffStanford}Data on the Stanford bunny (left) and
dominant samplet coefficients (right).}
\end{center}
\end{figure}

\section{Compression of kernel matrices}\label{sec:kernelCompression}
\subsection{Kernel matrices}
The second application of samplets we consider
is the compression of matrices arising from positive
(semi-) definite kernels, as they emerge in kernel 
methods, such as scattered data analysis, kernel 
based learning or Gaussian process regression,
see for example \cite{HSS08,Schaback2006,
Wendland2004,Williams1998} and the references 
therein.

We start by recalling the concept of a positive kernel.

\begin{definition}\label{def:poskernel}
A symmetric kernel
$\kernel\colon\Omega\times\Omega\rightarrow\Rbb$ is 
called \textit{positive (semi-)definite} on $\Omega\subset\R^d$,
iff \([\kernel({\bs x}_i,{\bs x}_j)]_{i,j=1}^N\)
is a symmetric and positive (semi-)definite matrix
for all
$\{{\bs x}_1, \ldots,{\bs x}_N\}\subset\Omega$
and all $N\in\mathbb{N}$.
\end{definition}

As a particular class of positive definite
kernels, we consider the \emph{Mat\'ern kernels} given by
\begin{equation}\label{eq:matkern}
k_\nu(r)\isdef\frac{2^{1-\nu}}{\Gamma(\nu)}\bigg(\frac {\sqrt{2\nu}r}
{\ell}\bigg)^\nu
K_\nu\bigg(\frac {\sqrt{2\nu}r}{\ell}\bigg),\quad r\geq 0,\ \ell >0 .
\end{equation}
Herein, $K_{\nu}$ is the modified Bessel function of the second 
kind of order $\nu$ and $\Gamma$ is the gamma function.
The parameter $\nu$ steers for the smoothness of the 
kernel function. Especially,
the analytic squared-exponential kernel is
retrieved for $\nu\to\infty$. Especially, we have
\begin{equation}
\begin{aligned}
k_{1/2}(r)=\exp\bigg(-\frac{r}{\ell}\bigg),
\quad k_{\infty}(r)=\exp\bigg(-\frac{r^2}{2\ell^2}\bigg).
\end{aligned}
\end{equation}
A positive definite kernel in the sense of
 Definition~\ref{def:poskernel}
is obtained by considering
\[
\kernel({\bs x},{\bs x}^\prime)\isdef
k_\nu(\|{\bs x}-{\bs x}^\prime\|_2).
\]

Given the set of points \(X=\{{\bs x}_1,\ldots,{\bs x}_N\}\), many
applications require the assembly and the inversion of the
\emph{kernel matrix}
\[
{\bs K}\isdef[\kernel({\bs x}_i,{\bs x}_j)]_{i,j=1}^N\in\Rbb^{N\times N}
\]
or an appropriately regularized version
\[
{\bs K}+\rho{\bs I},\quad \rho>0,
\]
thereof. In case that 
\(N\) is a large number, already the assembly and storage of
\({\bs K}\) 
can easily become prohibitive. For the solution of an associated 
linear system, the situation is even worse. 
Fortunately, the kernel matrix can be compressed 
by employing samplets. To this end, the evaluation of 
the kernel function at the points ${\bs x}_i$ and ${\bs x}_j$ 
will be denoted by 
\[
(\kernel,\delta_{{\bs x}_i}\otimes\delta_{{\bs x}_j}
)_{\Omega\times\Omega}\isdef\kernel({\bs x}_i,{\bs x}_j).
\] 
Hence, in view of $V = \{\delta_{{\bs x}_1},\ldots,\delta_{{\bs x}_N}\}$, 
we may write the kernel matrix as
\[
  {\bs K} = \big[(\kernel,\delta_{{\bs x}_i}
\otimes\delta_{{\bs x}_j})_{\Omega\times\Omega}\big]_{i,j=1}^N.
\]

\subsection{Asymptotically smooth kernels}
The essential ingredient for the samplet compression of 
kernel matrices is the \emph{asymptotical smoothness}
property of the kernel
\begin{equation}\label{eq:kernel_estimate}
  \frac{\partial^{|\bs\alpha|+|\bs\beta|}}
  	{\partial{\bs x}^{\bs\alpha}
  	\partial{\bs y}^{\bs\beta}} \kernel({\bs x},{\bs y})
  		\le c_\kernel \frac{(|\bs\alpha|+|\bs\beta|)!}
		{r^{|\bs\alpha|+|\bs\beta|}
		\|{\bs x}-{\bs y}\|_2^{|\bs\alpha|+|\bs\beta|}},\quad 
		c_\kernel,r>0,
\end{equation}
which is for example satisfied by the Mat\'ern kernels.
Using this estimate, we obtain the following result, 
which is the basis for the matrix compression introduced 
thereafter.

\begin{lemma}\label{lem:kernel_decay}
Consider two samplets $\sigma_{j,k}$ and $\sigma_{j',k'}$,
exhibiting $q+1$ vanishing moments with supporting
clusters \(\nu\) and \(\nu'\), respectively.
Assume that $\dist(\nu,\nu') > 0$. Then, for kernels 
satisfying \eqref{eq:kernel_estimate}, it holds that
\begin{equation}\label{eq:kernel_decay}
  (\kernel,\sigma_{j,k}\otimes\sigma_{j',k'})_{\Omega\times\Omega}\le
  	c_\kernel \frac{\diam(\nu)^{q+1}\diam(\nu')^{q+1}}
		{(dr\dist(\nu_{j,k},\nu_{j',k'}))^{2(q+1)}}
		\|{\bs\omega}_{j,k}\|_{1}\|{\bs\omega}_{j',k'}\|_{1}.
\end{equation}
\end{lemma}

\begin{proof}
Let ${\bs x}_0\in\nu$ and ${\bs y}_0\in\nu'$.
A Taylor expansion of the kernel with respect to 
${\bs x}$ yields
\[
\kernel({\bs x},{\bs y}) = \sum_{|\bs\alpha|\le q}
\frac{\partial^{|\bs\alpha|}}
  	{\partial{\bs x}^{\bs\alpha}\kernel({\bs x}_0,{\bs y})}
		\frac{({\bs x}-{\bs x}_0)^{\bs\alpha}}{\bs\alpha!}
			+ R_{{\bs x}_0}({\bs x},{\bs y}),
\]
where the remainder $R_{{\bs x}_0}({\bs x},{\bs y})$ is given by
\begin{align*}
  R_{{\bs x}_0}({\bs x},{\bs y}) &= (q+1)\sum_{|\bs\alpha|=q+1}
  	\frac{({\bs x}-{\bs x}_0)^{\bs\alpha}}{\bs\alpha!}
	\int_0^1\frac{\partial^{q+1}}{\partial{\bs x}^{\bs\alpha}}
	\kernel\big({\bs x}_0+s({\bs x}-{\bs x}_0),{\bs y}\big)(1-s)^q\d s.
\end{align*}
Next, we expand the remainder $R_{{\bs x}_0}({\bs x},{\bs y})$ with 
respect to ${\bs y}$ and derive
\begin{align*}
  R_{{\bs x}_0}({\bs x},{\bs y}) &= (q+1)\sum_{|\bs\alpha|=q+1}
  	\frac{({\bs x}-{\bs x}_0)^{\bs\alpha}}{\bs\alpha!}
	\sum_{|\bs\beta|\le q
	}\frac{({\bs y}-{\bs y}_0)^{\bs\beta}}{\bs\beta!}\\
	&\times\int_0^1\frac{\partial^{q+1}}{\partial{\bs x}^{\bs\alpha}}
	\frac{\partial^{|\bs\beta|}}{\partial{\bs y}^{\bs\beta}}
	\kernel\big({\bs x}_0+s({\bs x}-{\bs x}_0),{\bs y}_0\big)(1-s)^q\d s
		+ R_{{\bs x}_0,{\bs y}_0}({\bs x},{\bs y}).
\end{align*}
Here, the remainder $R_{{\bs x}_0,{\bs y}_0}({\bs x},{\bs y})$ 
is given by
\begin{align*}
  &R_{{\bs x}_0,{\bs y}_0}({\bs x},{\bs y}) = (q+1)^2 
  	\sum_{|\bs\alpha|,|\bs\beta|=q+1}
  	\frac{({\bs x}-{\bs x}_0)^{\bs\alpha}}{\bs\alpha!}
	\frac{({\bs y}-{\bs y}_0)^{\bs\beta}}{\bs\beta!}\\
	&\qquad\times\int_0^1\int_0^1\frac{\partial^{2(q+1)}}
  	{\partial{\bs x}^{\bs\alpha}\partial{\bs y}^{\bs\beta}}
	\kernel\big({\bs x}_0+s({\bs x}-{\bs x}_0),{\bs y}_0
	+t({\bs y}-{\bs y}_0)\big)(1-s)^q(1-t)^q\d t\d s.
\end{align*}
We thus arrive at the decomposition
\[
\kernel({\bs x},{\bs y}) = p_{{\bs y}}({\bs x}) + p_{{\bs x}}({\bs y}) 
	+ R_{{\bs x}_0,{\bs y}_0}({\bs x},{\bs y}),
\]
where $p_{{\bs y}}({\bs x})$ is a polynomial of degree $q$ in ${\bs x}$,
with coefficients depending on ${\bs y}$, while $p_{{\bs x}}({\bs y})$
is a polynomial of degree $q$ in ${\bs y}$, with coefficients depending 
on ${\bs x}$. Due to the vanishing moments, we obtain
\[
  (\kernel,\sigma_{j,k}\otimes\sigma_{j',k'})_{\Omega\times\Omega}
  	=(R_{{\bs x}_0,{\bs y}_0},
		\sigma_{j,k}\otimes\sigma_{j',k'})_{\Omega\times\Omega}.
\] 
In view of \eqref{eq:kernel_estimate}, we thus find
\begin{align*}
&|(\kernel,\sigma_{j,k}\otimes\sigma_{j',k'})_{\Omega\times\Omega}|
= |(R_{{\bs x}_0,{\bs y}_0},
		\sigma_{j,k}\otimes\sigma_{j',k'})_{\Omega\times\Omega}|\\
&\qquad\le c_\kernel \Bigg(\sum_{|\bs\alpha|,|\bs\beta|=q+1}
	\frac{(|\bs\alpha|+|\bs\beta|)!}{\bs\alpha!\bs\beta!}\Bigg)
	\frac{(\|\cdot-{\bs x}_0\|^{q+1}_2,|\sigma_{j,k}|)_\Omega
	(\|\cdot-{\bs y}_0\|^{q+1}_2,|\sigma_{j',k'}|)_\Omega}{
	r^{2(q+1)}\dist(\nu,\nu')^{2(q+1)}}.
\end{align*}
Next, we have by means of multinomial coefficients that
\[
(|\bs\alpha|+|\bs\beta|)!
={|\bs\alpha|+|\bs\beta|\choose |\bs\beta|}
{|\bs\alpha|\choose\bs\alpha}
{|\bs\beta|\choose\bs\beta}
\bs\alpha!\bs\beta!,
\]
which in turn implies that
\begin{align*}
  \sum_{|\bs\alpha|,|\bs\beta|=q+1}
  \frac{(|\bs\alpha|+|\bs\beta|)!}{\bs\alpha!\bs\beta!}
  &= {2(q+1)\choose q+1} \sum_{|\bs\alpha|,|\bs\beta|=q+1}
{|\bs\alpha|\choose\bs\alpha}
{|\bs\beta|\choose\bs\beta}\\
&= {2(q+1)\choose q+1} d^{2(q+1)}
\le d^{2(q+1)} 2^{2(q+1)}.
\end{align*}
Moreover, we use
\[
(\|\cdot-{\bs x}_0\|_2^{q+1},|\sigma_{j,k}|)_\Omega
\le\bigg(\frac{\diam(\nu)}{2}\bigg)^{q+1}\|{\bs\omega}_{j,k}\|_{1},
\]
and likewise
\[
(\|\cdot-{\bs y}_0\|_2^{q+1},|\sigma_{j',k'}|)_\Omega
\le\bigg(\frac{\diam(\nu')}{2}\bigg)^{q+1}\|{\bs\omega}_{j',k'}\|_{1}.
\]
Combining all the estimates, we arrive at the desired 
result \eqref{eq:kernel_decay}.
\end{proof}

\subsection{Matrix compression}
Lemma~\ref{lem:kernel_decay} immediately suggests 
a compression strategy for kernel matrices in
samplet representation. We mention that this compression 
differs from the wavelet matrix compression introduced
in \cite{DHS}, since we do not exploit the decay of the
samplet coefficients with respect to the level in case of 
smooth data. This enables us to also consider a non-uniform 
distribution of the points in $V$. Consequently, we use 
on all levels the same accuracy, what is more similar 
to the setting in \cite{BCR}.

\begin{theorem}\label{thm:compression}
Set all coefficients of the kernel matrix
\[
{\bs K}^\Sigma\isdef\big[(\kernel,\sigma_{j,k}\otimes\sigma_{j',k'})_{\Omega\times\Omega}
\big]_{j,j',k,k'}
\]
to zero which satisfy
\begin{equation}\label{eq:cutoff}
   \dist(\nu,\nu')\ge\eta\max\{\diam(\nu),\diam(\nu')\},\quad\eta>0,
\end{equation}
where \(\nu\) is the cluster supporting \(\sigma_{j,k}\) and
\(\nu'\) is the cluster supporting \(\sigma_{j',k'}\), respectively.
Then, it holds
\[
  \big\|{\bs K}^\Sigma-{\bs K}^\Sigma_\varepsilon\big\|_F
	\le c_\kernel c_{\operatorname{sum}} {(\eta dr)^{-2(q+1)}} 
		m_q N\sqrt{\log(N)}.
\]
for some constant \(c_{\operatorname{sum}}>0\),
where \(m_q\) is given by \eqref{eq:mq}.
\end{theorem}

\begin{proof}
We first fix the levels $j$ and $j'$. In view 
\eqref{eq:kernel_decay}, we can estimate any coefficient 
which satisfies \eqref{eq:cutoff} by
\begin{align*}
&|(\kernel,\sigma_{j,k}\otimes\sigma_{j',k'})_{\Omega\times\Omega}|\\
&\qquad\le
  	c_\kernel \bigg(\frac{\min\{\diam(\nu),\diam(\nu')\}}
		{\max\{\diam(\nu),\diam(\nu')\}}\bigg)^{q+1}
			{(\eta dr)^{-2(q+1)}}\|{\bs\omega}_{j,k}\|_{1}\|
			{\bs\omega}_{j',k'}\|_{1}.
\end{align*}
If we next set
\[
\theta_{j,j'}\isdef \max_{\nu\in\Tcal_j,\nu'\in\Tcal_{j'}}
\bigg\{\frac{\min\{\diam(\nu),\diam(\nu')\}}
{\max\{\diam(\nu),\diam(\nu')\}}\bigg\},
\]
then we obtain 
\[
|(\kernel,\sigma_{j,k}\otimes\sigma_{j',k'})_{\Omega\times\Omega}|
\le c_\kernel\theta_{j,j'}^{q+1}{(\eta dr)^{-2(q+1)}}
	\|{\bs\omega}_{j,k}\|_{1}\|{\bs\omega}_{j',k'}\|_{1}
\]
for all coefficients such that \eqref{eq:cutoff} holds. 
In view of \eqref{eq:ell1-norm} and the fact that there are 
at most $m_q$ samplets
per cluster, we arrive at
\[
  \sum_{k,k'} \|{\bs\omega}_{j,k}\|_{1}^2\|{\bs\omega}_{j',k'}\|_{1}^2
  	\leq\sum_{k,k'}|\nu|\cdot|\nu'| = m_q^2 N^2.
\]
Thus, for a fixed level-level block, we arrive at the estimate
\begin{align*}
 \big\|{\bs K}^\Sigma_{j,j'}-{\bs K}^\Sigma_{\varepsilon,j,j'}\big\|_F^2
 &\le\sum_{\begin{smallmatrix}k,k':\ \dist(\nu,\nu')\\
 \ge\eta\max\{\diam(\nu),\diam(\nu')\}\end{smallmatrix}}
  |(\kernel,\sigma_{j,k}\otimes\sigma_{j',k'})_{\Omega\times\Omega}|^2\\
  &\le c_\kernel^2 \theta_{j,j'}^{2(q+1)} {(\eta dr)^{-4(q+1)}}
  m_q^2 N^2.
\end{align*}
Finally, summation over all levels yields
\begin{align*}
  \big\|{\bs K}^\Sigma-{\bs K}^\Sigma_{\varepsilon}\big\|_F^2
  	&= \sum_{j,j'}\big\|{\bs K}^\Sigma_{j,j'}
  	-{\bs K}^\Sigma_{\varepsilon,j,j'}\big\|_F^2\\
	&\le c_\kernel^2 {(\eta dr)^{-4(q+1)}}m_q^2 N^2
	\sum_{j,j'} \theta_{j,j'}^{2(q+1)}\\
	&\le c_\kernel^2 c_{\operatorname{sum}} {(\eta dr)^{-4(q+1)}}
	m_q^2 N^2\log N,
\end{align*}
which is the desired claim.
\end{proof}

\begin{corollary}
In case of uniformly distributed points ${\bs x}_i\in X$, 
we have $\big\|{\bs K}^\Sigma\big\|_F\sim N$. Thus, 
we immediately obtain
\[
  \frac{\big\|{\bs K}^\Sigma-{\bs K}^\Sigma_\varepsilon\big\|_F}
  {\big\|{\bs K}^\Sigma\big\|_F} \le c_\kernel
  \sqrt{c_{\operatorname{sum}}} {(\eta dr)^{-2(q+1)}} m_q \sqrt{\log N}.
\]
In particular, the matrix can be compressed to
$\mathcal{O}(m_q^2
N\log N)$ remaining coefficients without compromising
the overall accuracy.
\end{corollary}

\begin{proof}
We fix $j,j'$ and assume $j\ge j'$. In case of uniformly 
distributed points, it holds $\diam(v)\sim 2^{-j_\nu/d}$.
Hence, for the cluster $\nu_{j',k'}$, there exist only
 $\mathcal{O}([2^{j-j'}]^d)$ clusters $\nu_{j,k}$ from
level $j$, which do not satisfy the cut-off criterion 
\eqref{eq:cutoff}. Since each cluster contains at most
$m_q$ samplets, we hence arrive at
\[
  \sum_{j=0}^J \sum_{j'\le j}m_q^2( 2^{j'} 2^{(j-j')})^d 
  = m_q^2 \sum_{j=0}^J j 2^{jd} \sim m_q^2 N\log N,
\]
which implies the assertion.
\end{proof}

\begin{remark}
The chosen cut-off criterion \eqref{eq:cutoff} coincides 
with the so called \emph{admissibility condition} used
by hierarchical matrices. We particularly refer here to
\cite{Boe10}, as we will later on rely the \(\mathcal{H}^2\)-matrix 
method presented there for the fast assembly of the 
compressed kernel matrix.
\end{remark}

\subsection{Compressed matrix assembly}
For a given pair of clusters, we can now determine whether the 
corresponding entries need to be calculated. As there are
$\mathcal{O}(N)$ clusters, naively checking the cut-off criterion for
all pairs would still take $\mathcal{O}(N^{2})$ operations, however.
Hence, we require smarter means to determine the non-negligible cluster
pairs. For this purpose, we first state the transferability of the
cut-off criterion to son clusters, compare \cite{DHS} for a proof.

\begin{lemma}
Let $\nu$ and $\nu'$ be clusters satisfying the cut-off criterion
\eqref{eq:cutoff}. Then, for the son clusters $\nu_{\mathrm{son}}$ 
of $\nu$ and $\nu_{\mathrm{son}}'$ of $\nu'$, we have 
\begin{align*}
\dist(\nu,\nu_{\mathrm{son}}')
&\ge\eta\max\{\diam(\nu),\diam(\nu_{\mathrm{son}}')\},\\
\dist(\nu_{\mathrm{son}},\nu')
&\ge\eta\max\{\diam(\nu_{\mathrm{son}}),\diam(\nu')\},\\
\dist(\nu_{\mathrm{son}},\nu_{\mathrm{son}}')
&\ge\eta\max\{\diam(\nu_{\mathrm{son}}),\diam(\nu_{\mathrm{son}}')\}.
\end{align*}
\end{lemma}

The lemma tells us that we may omit cluster pairs whose father 
clusters already satisfy the cut-off criterion. This will be essential
for the assembly of the compressed matrix.

The computation of the compressed kernel matrix
can be sped up further by using 
\(\Hcal^2\)-matrix techniques, see 
\cite{HB02,Gie01}. Similarly to \cite{AHK14,HKS05,Kae07}, we shall
rely here on \(\Hcal^2\)-matrices for this purpose.
The idea of \(\Hcal^2\)-matrices is to approximate the kernel
interaction
for sufficiently distant clusters \(\nu\) and \(\nu'\) in the sense 
of the admissibility condition \eqref{eq:cutoff} by means
of the interpolation based \(\Hcal^2\)-matrix approach.
More precisely, given a suitable set of interpolation 
points \(\{{\bs\xi}_t^\nu\}_t\) for each cluster \(\nu\) with 
associated Lagrange polynomials \(\{\mathcal{L}_{t}^{\nu}
({\bs x})\}_t\), we introduce the interpolation operator
\[
 \mathcal{I}^{\nu,\nu'}[\kernel]({\bs x}, {\bs y}) 
  	= \sum_{s,t} \kernel({\bs\xi}_{s}^{\nu}, {\bs\xi}_{t}^{\nu'}) 
		\mathcal{L}_{s}^{\nu}({\bs x}) \mathcal{L}_{t}^{\nu'}({\bs y})
\]
and approximate an admissible matrix block via
\begin{align*}
{\bs K}^\Delta_{\nu,\nu'}
&=[(\kernel,\delta_{\bs x}\otimes
\delta_{\bs y})_{\Omega\times\Omega}]_{{\bs x}\in\nu,{\bs y}\in\nu'}\\
&\approx\sum_{s,t} \kernel({\bs\xi}_{s}^{\nu}, {\bs\xi}_{t}^{\nu'}) 
		[(\mathcal{L}_{s}^{\nu},\delta_{\bs x})_\Omega]_{{\bs x}\in\nu}
		[(\mathcal{L}_{t}^{\nu'},\delta_{\bs y})_\Omega]_{{\bs y}\in\nu'}
		\defis{\bs V}^{\nu}_\Delta{\bs S}^{\nu,\nu'}
		({\bs V}^{\nu'}_\Delta)^\intercal.
\end{align*}
Herein, the \emph{cluster bases} are given according to
\begin{equation}\label{eq:cluster bases}
{\bs V}^{\nu}_\Delta\isdef [(\mathcal{L}_{s}^{\nu},
\delta_{\bs x})_\Omega]_{{\bs x}\in\nu},\quad
{\bs V}^{\nu'}_\Delta\isdef[(\mathcal{L}_{t}^{\nu'},
\delta_{\bs y})_\Omega]_{{\bs y}\in\nu'},
\end{equation}
while the \emph{coupling matrix} is given by
\(
{\bs S}^{\nu,\nu'}\isdef[\kernel({\bs\xi}_{s}^{\nu},
{\bs\xi}_{t}^{\nu'})]_{s,t}.
\)

Directly transforming the cluster bases into their corresponding 
samplet representation results in a log-linear cost. This can be 
avoided by the use of nested cluster bases, as they have been 
introduced for \(\Hcal^2\)-matrices. For the sake of simplicity, we
assume from now on that tensor product polynomials of degree 
\(p\) are used for the kernel interpolation at all different cluster 
combinations. As a consequence, the Lagrange polynomials 
of a father cluster can exactly be represented by those of the
son clusters. Introducing the \emph{transfer matrices}
\(
{\bs T}^{\nu_{\mathrm{son}}}
\isdef[\mathcal{L}_s^\nu({\bs\xi}_t^{\nu_{\mathrm{son}}})]_{s,t},
\)
there holds
\[
\mathcal{L}_s^\nu({\bs x})=\sum_t{\bs T}^{\nu_{\mathrm{son}}}_{s,t}
\mathcal{L}_t^{\nu_{\mathrm{son}}}({\bs x}),\quad{\bs x}
\in B_{\nu_{\mathrm{son}}}.
\]
Exploiting this relation in the construction of the cluster bases 
\eqref{eq:cluster bases} finally leads to
\[
{\bs V}^{\nu}_\Delta=\begin{bmatrix}
{\bs V}^{\nu_{\mathrm{son}_1}}_\Delta{\bs T}^{\nu_{\mathrm{son}_1}}\\
{\bs V}^{\nu_{\mathrm{son}_2}}_\Delta{\bs T}^{\nu_{\mathrm{son}_2}}
\end{bmatrix}.
\]

Combining this refinement relation with the recursive nature of the
samplet basis, results
in the variant of the discrete samplet transform summarized in
Algorithm~\ref{algo:multiscaleClusterBasis}.\bigskip

\begin{center}
\scalebox{0.8}{
\begin{algorithm}[H]
\caption{Recursive computation of the multiscale cluster basis}
\label{algo:multiscaleClusterBasis}
\KwData{Cluster tree $\Tcal$, transformations 
$[{\bs Q}_{j,\Phi}^{\nu}$, ${\bs Q}_{j,\Sigma}^{\nu}]$,
nested cluster bases ${\bs V}_{\Delta}^{\nu}$ for leaf clusters and
transformation matrices ${\bs T}^{\nu_{\mathrm{son}_1}}\), \(
{\bs T}^{\nu_{\mathrm{son}_2}}$ for non-leaf clusters.
}
\KwResult{Multiscale cluster basis matrices ${\bs V}_{\Phi}^{\nu}$, 
${\bs V}_{\Sigma}^{\nu}$ for all clusters $\nu \in\Tcal$.}
	
\Begin{
\FuncSty{computeMultiscaleClusterBasis}($X$)\;
}
\end{algorithm}}
\end{center}
\begin{center}
\scalebox{0.8}{
\begin{function}[H]
\caption{computeMultiscaleClusterBasis($\nu$)}
\Begin{
\uIf{$\nu$ is a leaf cluster}{
store $\begin{bmatrix}
{\bs V}_{\Phi}^{\nu} \\
{\bs V}_{\Sigma}^{\nu}
\end{bmatrix}\isdef\big[{\bs Q}_{j,\Phi}^{\nu}, 
{\bs Q}_{j,\Sigma}^{\nu}\big]^{\intercal} {\bs V}_{\Delta}^{\nu}$
		}
		\Else{
			\For{all sons $\nu'$ of $\nu$}{
				$\computeMultiscaleClusterBasis(\nu')$
			}
			store $\begin{bmatrix}
				{\bs V}_{\Phi}^{\nu} \\
				{\bs V}_{\Sigma}^{\nu}
			\end{bmatrix}\isdef\big[{\bs Q}_{j,\Phi}^{\nu}, 
			{\bs Q}_{j,\Sigma}^{\nu}\big]^{\intercal} \begin{bmatrix}
{\bs V}^{\nu_{\mathrm{son}_1}}_\Phi{\bs T}^{\nu_{\mathrm{son}_1}}\\
{\bs V}^{\nu_{\mathrm{son}_2}}_\Phi{\bs T}^{\nu_{\mathrm{son}_2}}
\end{bmatrix}$
		}
	}
\end{function}}
\end{center}

Having the multiscale cluster bases at our disposal, the next step is
the assembly of the compressed kernel matrix. The computation of the
required matrix blocks is exclusively 
based on the two refinement relations
\[
\begin{bmatrix}
{\bs K}_{\nu,\nu'}^{\Phi,\Phi} & {\bs K}_{\nu,\nu'}^{\Phi,\Sigma}\\
{\bs K}_{\nu,\nu'}^{\Sigma,\Phi} & {\bs K}_{\nu,\nu'}^{\Sigma,\Sigma}
\end{bmatrix}
=\begin{bmatrix}
{\bs K}_{\nu,\nu_{\mathrm{son}_1}'}^{\Phi,\Phi} &
{\bs K}_{\nu,\nu_{\mathrm{son}_2}'}^{\Phi,\Phi}\\
{\bs K}_{\nu,\nu_{\mathrm{son}_1}'}^{\Sigma,\Phi} &
{\bs K}_{\nu,\nu_{\mathrm{son}_2}'}^{\Sigma,\Phi}
\end{bmatrix}\big[{\bs Q}_{j,\Phi}^{\nu'}, 
			{\bs Q}_{j,\Sigma}^{\nu'}\big]
\]
and
\[
\begin{bmatrix}
{\bs K}_{\nu,\nu'}^{\Phi,\Phi} & {\bs K}_{\nu,\nu'}^{\Phi,\Sigma}\\
{\bs K}_{\nu,\nu'}^{\Sigma,\Phi} & {\bs K}_{\nu,\nu'}^{\Sigma,\Sigma}
\end{bmatrix}
=
\big[{\bs Q}_{j,\Phi}^{\nu}, 
			{\bs Q}_{j,\Sigma}^{\nu}\big]^\intercal
			\begin{bmatrix}
{\bs K}_{\nu_{\mathrm{son}_1},\nu'}^{\Phi,\Phi} &
{\bs K}_{\nu_{\mathrm{son}_1},\nu'}^{\Phi,\Phi}\\
{\bs K}_{\nu_{\mathrm{son}_2},\nu'}^{\Sigma,\Phi} &
{\bs K}_{\nu_{\mathrm{son}_2},\nu'}^{\Sigma,\Phi}
\end{bmatrix},
\]
where we set
\[
\begin{bmatrix}
{\bs K}_{\nu,\nu'}^{\Phi,\Phi} & {\bs K}_{\nu,\nu'}^{\Phi,\Sigma}\\
{\bs K}_{\nu,\nu'}^{\Sigma,\Phi} & {\bs K}_{\nu,\nu'}^{\Sigma,\Sigma}
\end{bmatrix}
\isdef
\begin{bmatrix}
(\kernel,{\bs \Phi}^\nu\otimes{\bs\Phi}^{\nu'})_{\Omega\times\Omega} &
(\kernel,{\bs \Phi}^\nu\otimes{\bs\Sigma}^{\nu'})_{\Omega\times\Omega}
\\
(\kernel,{\bs \Sigma}^\nu\otimes{\bs\Phi}^{\nu'})_{\Omega\times\Omega}
&
(\kernel,{\bs \Sigma}^\nu\otimes{\bs\Sigma}^{\nu'})_{\Omega\times\Omega} 
\end{bmatrix}.
\]

We obtain the following function, which is the key ingredient for the
computation of the compressed kernel matrix.

\begin{center}
\scalebox{0.8}{
\begin{function}[H]
\caption{recursivelyDetermineBlock($\nu$, $\nu'$)}
\KwResult{Approximation of the block \scalebox{1}{$\begin{bmatrix}
{\bs K}_{\nu,\nu'}^{\Phi,\Phi} & {\bs K}_{\nu,\nu'}^{\Phi,\Sigma}\\
{\bs K}_{\nu,\nu'}^{\Sigma,\Phi} & {\bs K}_{\nu,\nu'}^{\Sigma,\Sigma}
\end{bmatrix}$}.}
	
\Begin{
\uIf{$(\nu, \nu')$ is admissible}{
\Return{\scalebox{1}{$\begin{bmatrix}
{\bs V}_{\Phi}^{\nu} \\
{\bs V}_{\Sigma}^{\nu}
\end{bmatrix}
{\bs S}^{\nu,\nu'} \big[
({\bs V}_{\Phi}^{\nu'})^\intercal,
({\bs V}_{\Sigma}^{\nu'})^\intercal
\big]$}}
}
\uElseIf{$\nu$ and $\nu'$ are leaf clusters}{
\Return{\scalebox{1}{$\big[{\bs Q}_{j,\Phi}^{\nu}, 
{\bs Q}_{j,\Sigma}^{\nu}\big]^{\intercal}{\bs K}_{\nu,\nu'}^{\Delta}
\big[{\bs Q}_{j,\Phi}^{\nu'}, 
{\bs Q}_{j,\Sigma}^{\nu'}\big]$}}
		}
		\uElseIf{$\nu'$ is not a leaf cluster and $\nu$ is a leaf cluster}{
			\For{all sons $\nu_{\mathrm{son}}'$ of $\nu'$}{
				\scalebox{1}{$\begin{bmatrix}
{\bs K}_{\nu,\nu_{\mathrm{son}}'}^{\Phi,\Phi} & 
{\bs K}_{\nu,\nu_{\mathrm{son}}'}^{\Phi,\Sigma}\\
{\bs K}_{\nu,\nu_{\mathrm{son}}'}^{\Sigma,\Phi} & 
{\bs K}_{\nu,\nu_{\mathrm{son}}'}^{\Sigma,\Sigma}
\end{bmatrix}$} $
\isdef\recursivelyDetermineBlock(\nu, \nu_{\mathrm{son}'})$
			}
			\Return{\scalebox{1}{$\begin{bmatrix}
{\bs K}_{\nu,\nu_{\mathrm{son},1}'}^{\Phi,\Phi} &
{\bs K}_{\nu,\nu_{\mathrm{son},2}'}^{\Phi,\Phi}\\
{\bs K}_{\nu,\nu_{\mathrm{son},1}'}^{\Sigma,\Phi} &
{\bs K}_{\nu,\nu_{\mathrm{son},2}'}^{\Sigma,\Phi}
\end{bmatrix}\big[{\bs Q}_{j,\Phi}^{\nu'}, 
{\bs Q}_{j,\Sigma}^{\nu'}\big]$}}
}
\uElseIf{$\nu$ is not a leaf cluster and $\nu'$ is a leaf cluster}{
\For{all sons \(\nu_{\mathrm{son}}\) of \(\nu\)}{
\scalebox{1}{$\begin{bmatrix}
{\bs K}_{\nu_{\mathrm{son}},\nu'}^{\Phi,\Phi} & 
{\bs K}_{\nu_{\mathrm{son}},\nu'}^{\Phi,\Sigma}\\
{\bs K}_{\nu_{\mathrm{son}},\nu'}^{\Sigma,\Phi} & 
{\bs K}_{\nu_{\mathrm{son}},\nu'}^{\Sigma,\Sigma}
\end{bmatrix}$} $\isdef
\recursivelyDetermineBlock(\nu_{\mathrm{son}}, \nu')$
}
\Return{\scalebox{1}{$\big[{\bs Q}_{j,\Phi}^{\nu}, 
{\bs Q}_{j,\Sigma}^{\nu}\big]^\intercal\begin{bmatrix}
{\bs K}_{\nu_{\mathrm{son}_1},\nu'}^{\Phi,\Phi} &
{\bs K}_{\nu_{\mathrm{son}_1},\nu'}^{\Phi,\Phi}\\
{\bs K}_{\nu_{\mathrm{son}_2},\nu'}^{\Sigma,\Phi} &
{\bs K}_{\nu_{\mathrm{son}_2},\nu'}^{\Sigma,\Phi}
\end{bmatrix}$}.
}
}
\Else(){
\For{all sons $\nu_{\mathrm{son}}$ of $\nu$ {\bf and} 
all sons $\nu_{\mathrm{son}}'$ of $\nu'$}{
\scalebox{1}{$\begin{bmatrix}
{\bs K}_{\nu_{\mathrm{son}},\nu_{\mathrm{son}}'}^{\Phi,\Phi} & 
{\bs K}_{\nu_{\mathrm{son}},\nu_{\mathrm{son}}'}^{\Phi,\Sigma}\\
{\bs K}_{\nu_{\mathrm{son}},\nu_{\mathrm{son}}'}^{\Sigma,\Phi} & 
{\bs K}_{\nu_{\mathrm{son}},\nu_{\mathrm{son}}'}^{\Sigma,\Sigma}
\end{bmatrix}$}$\isdef\recursivelyDetermineBlock(\nu_{\mathrm{son}},
\nu_{\mathrm{son}'})$
}
\Return{\scalebox{1}{$\big[{\bs Q}_{\Phi}^{\nu}, 
{\bs Q}_{\Sigma}^{\nu}\big]^\intercal\begin{bmatrix}
{\bs K}_{\nu_{\mathrm{son}_1},\nu_{\mathrm{son}_1}'}^{\Phi,\Phi} &
{\bs K}_{\nu_{\mathrm{son}_1},\nu_{\mathrm{son}_2}'}^{\Phi,\Phi}\\
{\bs K}_{\nu_{\mathrm{son}_2},\nu_{\mathrm{son}_1}'}^{\Phi,\Phi} &
{\bs K}_{\nu_{\mathrm{son}_2},\nu_{\mathrm{son}_2}'}^{\Phi,\Phi}
\end{bmatrix} \big[{\bs Q}_{\Phi}^{\nu'}, 
{\bs Q}_{\Sigma}^{\nu'}\big]$}}
}
}
\end{function}}
\end{center}
\textcolor{red}{We remark that the algorithm never requires the formation of the
entire \(\mathcal{H}^2\)-matrix, as it only embeds the multilevel
interpolation procedure to rapidly evaluate admissible blocks.
In particular, the evaluation of the
coupling matrices can be performed on the fly.}

Now, in order to assemble the compressed kernel matrix, we require
two nested recursive calls of the cluster tree, which is traversed in
a depth first search way. Algorithm~\ref{algo:h2Wavelet}
first computes the lower right matrix block and advances from bottom
to top and from right to left. To this end, the two recursive
functions \texttt{setupColumn} and \texttt{setupRow} are introduced.%

\begin{center}
\scalebox{0.8}{
\begin{algorithm}[H]
	\caption{Computation of the compressed kernel matrix}
	\label{algo:h2Wavelet}	
	\KwData{Cluster tree $\Tcal$, multiscale cluster bases
	${\bs V}_{\Phi}^{\nu}$, ${\bs V}_{\Sigma}^{\nu}$
	and transformations $[{\bs Q}_{j,\Phi}^{\nu},
	{\bs Q}_{j,\Sigma}^{\nu}]$.}
	\KwResult{Sparse matrix ${\bs K}^\Sigma_\varepsilon$}
	
	\Begin{
		\FuncSty{setupColumn}($X$)
		
		store the blocks the remaining blocks
		${\bs K}^\Sigma_{\varepsilon,\nu,X}$ for \(%
\nu\in\Tcal\setminus\{X\}\)   
		in ${\bs K}^\Sigma_\varepsilon$ (they have already
		 been computed by earlier calls to %
		  \FuncSty{recursivelyDetermineBlock})
	}
\end{algorithm}}
\end{center}

The purpose of the function \texttt{setupColumn} is to
recursively traverse the column cluster tree, i.e.\ the
cluster tree associated to the columns of the matrix.
Before returning, each instance of \texttt{setupColumn}
calls the function \texttt{setupRow}, which performs
the actual assembly of the compressed matrix.\bigskip

\begin{center}
\scalebox{0.8}{
\begin{function}[H]
	\caption{setupColumn($\nu'$)}
	
	\Begin{
		\For{all sons $\nu_{\mathrm{son}}'$ of $\nu'$}{
			$\setupColumn(\nu_{\mathrm{son}}')$
		}
		store ${\bs K}^\Sigma_{\varepsilon,X,\nu'}\isdef 
		\FuncSty{setupRow}(X, \nu')$
		in ${\bs K}^\Sigma_{\varepsilon}$
	}
\end{function}}
\end{center}

For a given column cluster \(\nu'\), the function
\texttt{setupRow}
recursively traverses the row cluster tree, i.e.\
the cluster tree associated to the rows of the matrix,
and
assembles the corresponding column of the compressed
 matrix.
The function reuses the already computed blocks to the
right of the column
under consideration and blocks at the bottom of the
very same
column.

\begin{center}
\scalebox{0.8}{
\begin{function}[H]
\caption{setupRow($\nu$, $\nu'$)}
\Begin{
\uIf{$\nu$ is not a leaf}{
\For{all sons \(\nu_{\mathrm{son}}\) of \(\nu\)}{
\uIf{\(\nu_{\mathrm{son}}\) and \(\nu'\) are not %
admissible}{
\scalebox{1}{$\begin{bmatrix}
{\bs K}_{\nu_{\mathrm{son}},\nu'}^{\Phi,\Phi} &
{\bs K}_{\nu_{\mathrm{son}},\nu'}^{\Phi,\Sigma}\\
{\bs K}_{\nu_{\mathrm{son}},\nu'}^{\Sigma,\Phi} &
{\bs K}_{\nu_{\mathrm{son}},\nu'}^{\Sigma,\Sigma}
\end{bmatrix}$}
$\isdef \setupRow(\nu_{\mathrm{son}}, \nu')$
				}
\Else{
\scalebox{1}{$\begin{bmatrix}
{\bs K}_{\nu_{\mathrm{son}},\nu'}^{\Phi,\Phi} &
{\bs K}_{\nu_{\mathrm{son}},\nu'}^{\Phi,\Sigma}\\
{\bs K}_{\nu_{\mathrm{son}},\nu'}^{\Sigma,\Phi} &
{\bs K}_{\nu_{\mathrm{son}},\nu'}^{\Sigma,\Sigma}
\end{bmatrix}$}
$\isdef\recursivelyDetermineBlock(\nu_{\mathrm{son}},%
\nu')$}
			}
\scalebox{1}{$
\begin{bmatrix}
{\bs K}_{\nu,\nu'}^{\Phi,\Phi} &
{\bs K}_{\nu,\nu'}^{\Phi,\Sigma}\\
{\bs K}_{\nu,\nu'}^{\Sigma,\Phi} &
{\bs K}_{\nu,\nu'}^{\Sigma,\Sigma}
\end{bmatrix}\isdef\big[{\bs Q}_{\Phi}^{\nu}, 
			{\bs Q}_{\Sigma}^{\nu}
\big]^\intercal\begin{bmatrix}
{\bs K}_{\nu_{\mathrm{son}_1},\nu'}^{\Phi,\Phi} &
{\bs K}_{\nu_{\mathrm{son}_1},\nu'}^{\Phi,\Phi}\\
{\bs K}_{\nu_{\mathrm{son}_2},\nu'}^{\Sigma,\Phi} &
{\bs K}_{\nu_{\mathrm{son}_2},\nu'}^{\Sigma,\Phi}
\end{bmatrix}$
}
}
\Else{
\uIf{$\nu'$ is a leaf cluster}{
\scalebox{1}{$\begin{bmatrix}
{\bs K}_{\nu_{\mathrm{son}},\nu'}^{\Phi,\Phi} & 
{\bs K}_{\nu_{\mathrm{son}},\nu'}^{\Phi,\Sigma}\\
{\bs K}_{\nu_{\mathrm{son}},\nu'}^{\Sigma,\Phi} &
{\bs K}_{\nu_{\mathrm{son}},\nu'}^{\Sigma,\Sigma}
\end{bmatrix}$}
$\isdef\recursivelyDetermineBlock(%
\nu_{\mathrm{son}}, \nu')$
}
\Else{
\For{all sons \(\nu_{\mathrm{son}}'\) of \(\nu\)'}{
\uIf{\(\nu\) and \(\nu_{\mathrm{son}}'\) %
are not admissible}{
load already computed block  \scalebox{1}{%
$\begin{bmatrix}
{\bs K}_{\nu,\nu_{\mathrm{son}}'}^{\Phi,\Phi} &
{\bs K}_{\nu,\nu_{\mathrm{son}}'}^{\Phi,\Sigma}\\
{\bs K}_{\nu,\nu_{\mathrm{son}}'}^{\Sigma,\Phi} &
{\bs K}_{\nu,\nu_{\mathrm{son}}'}^{\Sigma,\Sigma}
\end{bmatrix}$}
}
\Else
{
 \scalebox{1}{$\begin{bmatrix}
{\bs K}_{\nu,\nu_{\mathrm{son}}'}^{\Phi,\Phi} & 
{\bs K}_{\nu,\nu_{\mathrm{son}}'}^{\Phi,\Sigma}\\
{\bs K}_{\nu,\nu_{\mathrm{son}}'}^{\Sigma,\Phi} & 
{\bs K}_{\nu,\nu_{\mathrm{son}}'}^{\Sigma,\Sigma}
\end{bmatrix}$}
$\isdef \recursivelyDetermineBlock(%
\nu, \nu_{\mathrm{son}'})$
}
}
}
\scalebox{1}{
$\begin{bmatrix}{\bs K}_{\nu,\nu'}^{\Phi,\Phi} &
{\bs K}_{\nu,\nu'}^{\Phi,\Sigma}\\
{\bs K}_{\nu,\nu'}^{\Sigma,\Phi} &
{\bs K}_{\nu,\nu'}^{\Sigma,\Sigma}
\end{bmatrix}\isdef\begin{bmatrix}
{\bs K}_{\nu,\nu_{\mathrm{son}_1}'}^{\Phi,\Phi} &
{\bs K}_{\nu,\nu_{\mathrm{son}_2}'}^{\Phi,\Phi}\\
{\bs K}_{\nu,\nu_{\mathrm{son}_1}'}^{\Sigma,\Phi} &
{\bs K}_{\nu,\nu_{\mathrm{son}_2}'}^{\Sigma,\Phi}
\end{bmatrix}\big[{\bs Q}_{\Phi}^{\nu'}, 
			{\bs Q}_{\Sigma}^{\nu'}\big]$}
}
store ${\bs K}_{\nu,\nu'}^{\Sigma,\Sigma}$ %
as part of ${\bs K}^\Sigma_\varepsilon$
		
\Return{\scalebox{1}{$\begin{bmatrix}
{\bs K}_{\nu,\nu'}^{\Phi,\Phi} &
{\bs K}_{\nu,\nu'}^{\Phi,\Sigma}\\
{\bs K}_{\nu,\nu'}^{\Sigma,\Phi} &
{\bs K}_{\nu,\nu'}^{\Sigma,\Sigma}
\end{bmatrix}$}}
	}
\end{function}}
\end{center}

\begin{remark}
Algorithm~\ref{algo:h2Wavelet} has a cost of
\(\mathcal{O}(N\log N)\) and requires an additional
storage of \(\mathcal{O}(N\log N)\) if all stored
blocks are directly released when they are not required
anymore. We refer to \cite{AHK14} for all the details.
\end{remark}

\section{Numerical results I\!I}\label{sec:Num2}
All computations in this section have been performed on
a single node with two Intel Xeon E5-2650 v3 @2.30GHz
CPUs and up to 512GB of main memory\footnote{The full
specifications can be found on
https://www.euler.usi.ch/en/research/resources.}.
In order to obtain consistent timings, only a single
 core was used for all computations.

\subsection*{Benchmark problem}
\textcolor{red}{To benchmark the compression of kernel matrices,
we consider the exponential kernel
\[
k({\bs x},{\bs y})=e^{-100
\|{\bs x}-{\bs y}\|_2},
\]
evaluated at an increasing number of non-uniformly 
distributed cloud of point samples. Namely, in $d=1$ 
dimension, we consider standard normally distributed 
points. In $d>1$ dimensions, the random sample points 
are drawn from the mixture of two multivariate Gaussian 
distributions with zero expectation covariances
\[
\begin{bmatrix}
1&   -1/2&    0\\
-1/2 & 29/100 &    0\\
0 &      0& 1
\end{bmatrix}\quad\text{and}\quad
\begin{bmatrix}
1&   1/2&    0\\
1/2 & 29/100 &    0\\
0 &      0& 1
\end{bmatrix}.
\]
Note that the last coordinate is dropped if \(d=2\). The
resulting data sets are visualized in Figure~\ref{fig:dataSets}
with the corresponding bounding boxes of the domain and of the tree leaves.
For \(d=2\), the bounding box is given by \([-5.21,4.56]\times[-2.50,2.48]\),
while it is given by \([-5.21,4.56]\times[-2.50,2.48]\times[-4.81,4.84]\) 
for \(d=3\). As can be seen, the points have a much higher density at the 
center of the point cloud, which results in an adaptively refined cluster tree.}

\begin{figure}[htb]
\begin{center}
\scalebox{0.8}{
\includegraphics[scale=0.08,trim = 220 150 220 150,clip]{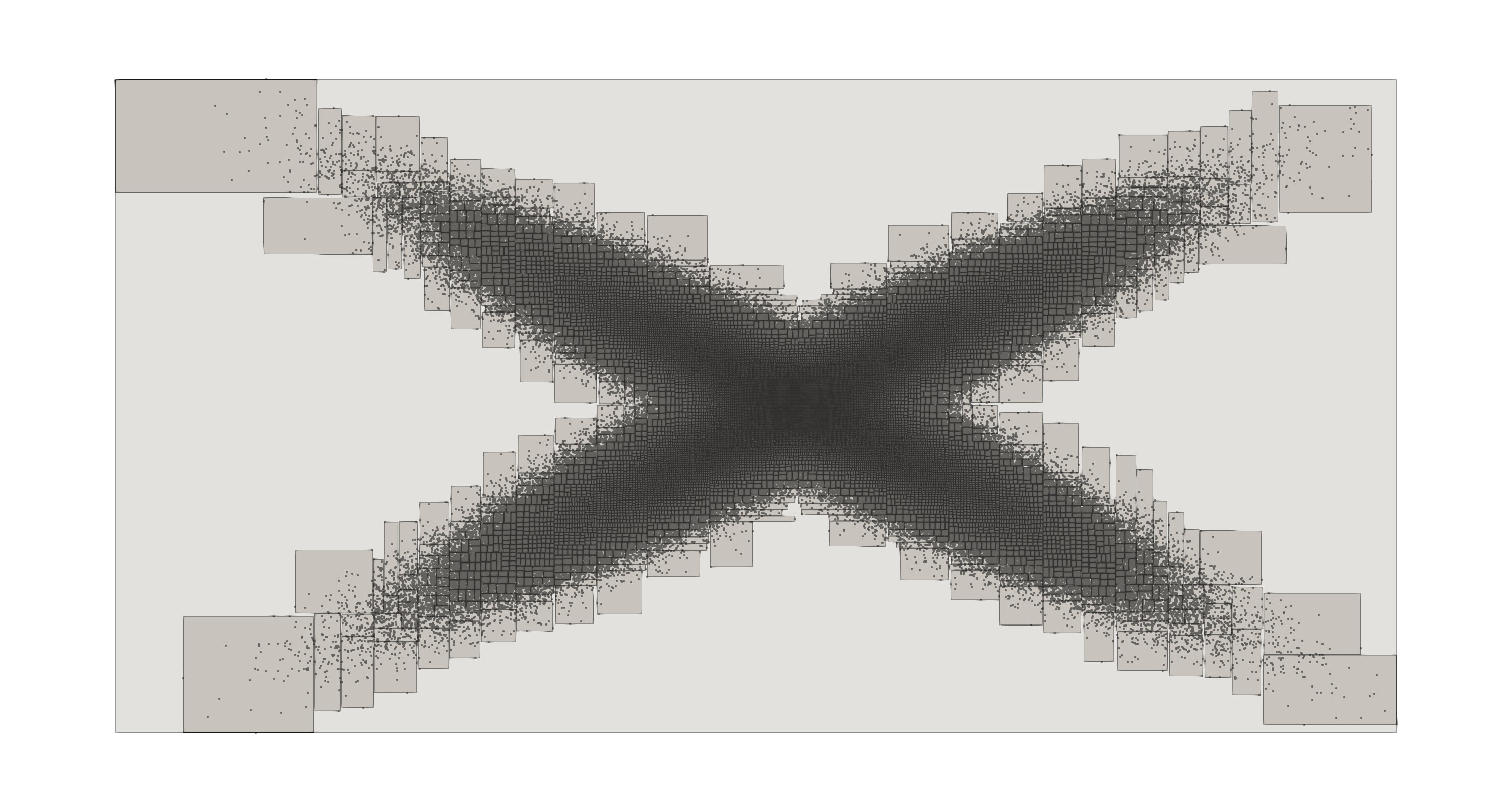}\quad
\includegraphics[scale=0.077,trim = 170 125 170 125,clip]{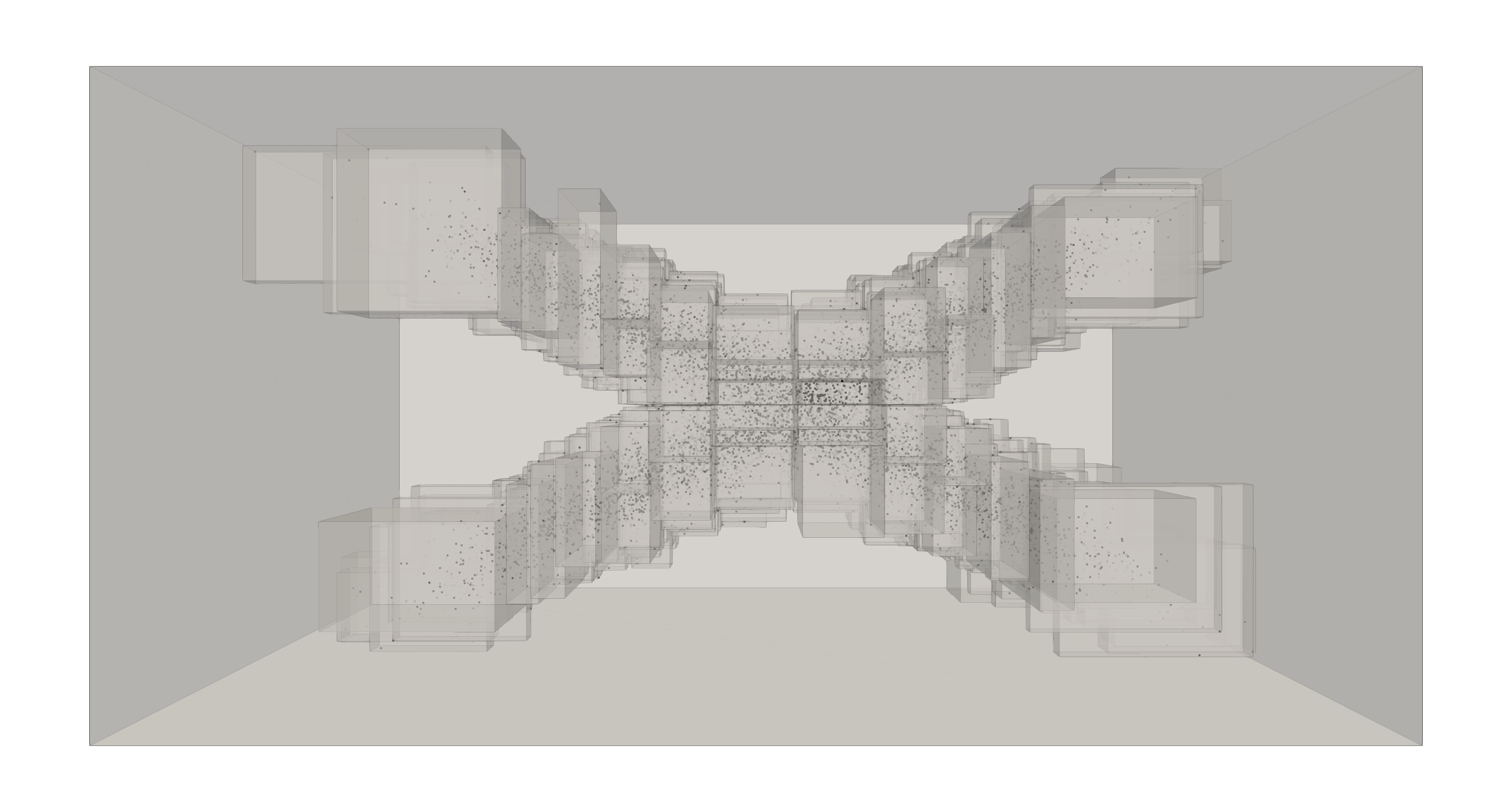}}
\caption{\label{fig:dataSets} Test data sets with bounding boxes for the domain
and for the tree leaves for \(d=2\) (left) and \(d=3\) (right).}
\end{center}
\end{figure}

As a measure of sparsity, we introduce the
\emph{average number of nonzeros per row}
\[
\operatorname{anz}({\bs A})
\isdef\frac{\operatorname{nnz}({\bs A})}{N},\quad
{\bs A}\in\Rbb^{N\times N},
\]
where \(\operatorname{nnz}({\bs A})\) is the number of
nonzero entries of \({\bs A}\). Besides the compression, 
we also report the fill-in generated by the 
Cholesky factorization in combination with the
nested dissection reordering from \cite{KK98}.
For the reordering and the Cholesky
factorization, we rely on \textsc{Matlab}%
 R2020a\footnote{Version 9.8.0.1396136, 
The MathWorks Inc., Natick, Massachusetts, 2020.},
 while the 
samplet compression is implemented in \texttt{C++11}
using the 
\texttt{Eigen} template
library\footnote{\texttt{https://eigen.tuxfamily.org/}}
for linear algebra operations. For the computations,
we consider
a polynomial degree of 3 for the kernel interpolation
and \(q+1=3\) vanishing moments for the samplets.
We set \(\eta=2\) for \(d=1\), \(\eta=1.25\) for \(d=2\)
and \(\eta=0.5\) for \(d=3\).
In addition, we have performed a thresholding of the
computed matrix  coefficients that were smaller than
\(\varepsilon=10^{-5}\).

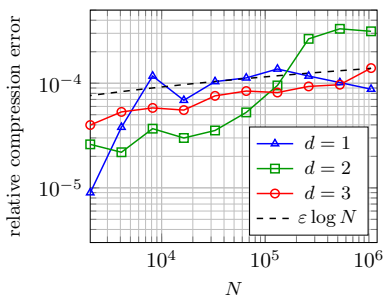
\begin{figure}[htb]
\begin{center}
\scalebox{0.8}{
\begin{tikzpicture}
\begin{loglogaxis}[width=0.42\textwidth,grid=both,
 ymin= 3e-6, ymax = 0.5e-3, xmin = 2048, xmax =1.2e6,
legend style={legend pos=south east,font=\small},
 ylabel={\small relative compression error}, xlabel ={\small $N$}]
\addplot[line width=0.7pt,color=blue,mark=triangle]
table[x=Npts,y=err1D]{./ResultsNew/compressionErrors.txt};
\addlegendentry{$d=1$};
\addplot[line width=0.7pt,color=darkgreen,mark=square]
table[x=Npts,y=err2D]{./ResultsNew/compressionErrors.txt};
\addlegendentry{$d=2$};
\addplot[line width=0.7pt,color=red,mark=o]
table[x=Npts,y=err3D]{./ResultsNew/compressionErrors.txt};
\addlegendentry{$d=3$};
\addplot[line width=0.7pt,color=black,dashed]
table[x=npts,y expr={1e-5 * ln(x)}]{%
./Results/matlabLogger1.txt};
\addplot[line width=0.7pt,color=black,dashed]
table[x=npts,y expr={0.1e-6 * x^2}]{%
./Results/matlabLogger1.txt};
\addlegendentry{$\varepsilon\log N$};
\end{loglogaxis}
\end{tikzpicture}}
\caption{\label{fig:compErrs}Relative compression errors for \(d=1,2,3\).}
\end{center}
\end{figure}

\textcolor{red}{Figure~\ref{fig:compErrs} shows the resulting relative 
compression errors, which have been computed by estimating the Frobenius 
norm from 20 randomly chosen columns of \({\bs K}^\Sigma\) and 
\({\bs K}^\Sigma_\varepsilon\), respectively. As can be seen, 
for all dimensions under consideration, the compression errors
roughly follow the theoretical rate of \(\varepsilon\log N\).}

\begin{figure}[htb]
\begin{center}
\scalebox{0.8}{
\begin{tikzpicture}
\begin{loglogaxis}[width=0.42\textwidth,grid=both,
 ymin= 2e-1, ymax = 0.2e5, xmin = 2048, xmax =1.2e6,
    legend style={legend pos=south east,font=\small},
     ylabel={\small wall time}, xlabel ={\small $N$}]
\addplot[line width=0.7pt,color=blue,mark=triangle]
table[x=npts,y=ctim]{./ResultsNew/matlabLogger1.txt};
\label{pgfplots:plot1D}
\addplot[line width=0.7pt,color=darkgreen,mark=square]
table[x=npts,y=ctim]{./ResultsNew/matlabLogger2.txt};
\label{pgfplots:plot2D}
\addplot[line width=0.7pt,color=red,mark=o]
table[x=npts,y=ctim]{./ResultsNew/matlabLogger3.txt};
\label{pgfplots:plot3D}
\addplot[line width=0.7pt,color=black, dashed]
table[x=npts,y expr={1e-4 * x}]{%
./ResultsNew/matlabLogger1.txt};
\addplot[line width=0.7pt,color=black, dashed]
table[x=npts,y expr={3e-5 * x * ln(x)}]{%
./ResultsNew/matlabLogger1.txt};
\addplot[line width=0.7pt,color=black, dashed]
table[x=npts,y expr={1e-5 * x * ln(x)^2}]{%
./ResultsNew/matlabLogger1.txt};
\addplot[line width=0.7pt,color=black, dashed]
table[x=npts,y expr={0.5e-5 * x * ln(x)^3}]{%
./ResultsNew/matlabLogger1.txt};
\label{pgfplots:asymps}
\end{loglogaxis}
\begin{loglogaxis}[%
xshift=0.405\textwidth,width=0.42\textwidth,grid=both,
ymin= 80, ymax = 3e3, xmin = 2048, xmax =1.2e6,
ytick={1e1, 1e2, 1e3, 1e4},
legend style={legend pos=south east,font=\small},
ylabel={\small $\operatorname{anz}({\bs K}%
^\Sigma_\varepsilon)$}, xlabel ={\small $N$}]
\addplot[line width=0.7pt,color=blue,%
mark=triangle] table[x=npts,
y expr = {\thisrow{nzS}}]{./ResultsNew/matlabLogger1.txt};
\addplot[line width=0.7pt,color=darkgreen,mark=square]%
 table[x=npts,
y expr = {\thisrow{nzS}}]{./ResultsNew/matlabLogger2.txt};
\addplot[line width=0.7pt,color=red,mark=o]%
 table[x=npts,
y expr = {\thisrow{nzS}}]{./ResultsNew/matlabLogger3.txt};
\end{loglogaxis}
\matrix[
    matrix of nodes,
    anchor=north west,
    draw,
    inner sep=0.1em,
    column 1/.style={nodes={anchor=center}},
    column 2/.style={nodes={anchor=west},font=\strut},
    draw
  ]
  at([xshift=0.02\textwidth]current axis.north east){
    \ref{pgfplots:plot1D}& \(d=1\)\\
    \ref{pgfplots:plot2D}& \(d=2\)\\
    \ref{pgfplots:plot3D}& \(d=3\)\\
    \ref{pgfplots:asymps}& \(N\!\log^\alpha\!\! N\)\\};
\end{tikzpicture}}
\caption{\label{fig:compTimesNNZ}Assembly times
(left) and average numbers of nonzeros per row (right)
versus the number sample points $N$ in case of the
 exponential kernel matrix.}
\end{center}
\end{figure}
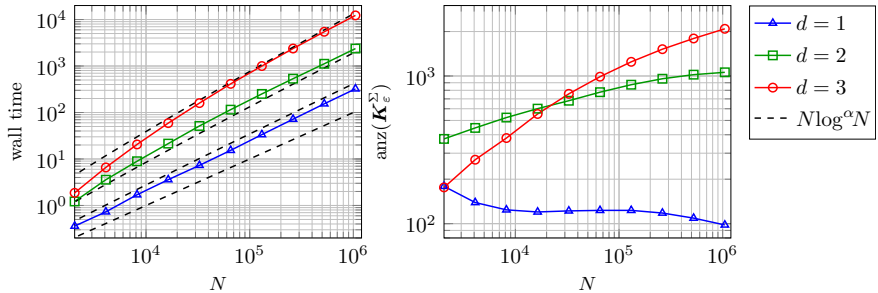

The left-hand side of Figure~\ref{fig:compTimesNNZ}
shows the wall time for the assembly of the compressed
kernel matrices. The different dashed lines indicate
the asymptotics \(N\log^\alpha N\) 
for \(\alpha=0,1,2,3\). For
increasing number 
\(N\) of points and the dimensions \(d=1,2,3\) under
consideration, all computation times approach the
 expected rate 
of \(N\log N\). The right-hand side of
Figure~\ref{fig:compTimesNNZ} shows the average number
of nonzeros per row for an increasing number 
\(N\) of points. This number becomes constant or even
slightly
decreases, as expected.

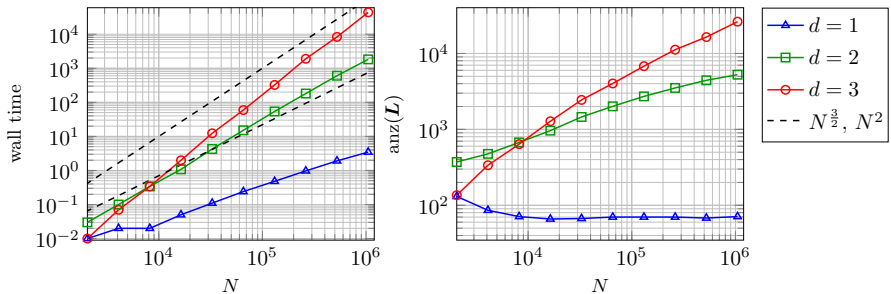
\begin{figure}[htb]
\begin{center}
\scalebox{0.8}{
\begin{tikzpicture}
\begin{loglogaxis}[width=0.42\textwidth,grid=both,
 ymin= 0.9e-2, ymax = 6e4, xmin = 2048, xmax =1.2e6,
legend style={legend pos=south east,font=\small},
 ylabel={\small wall time}, xlabel ={\small $N$},
 ytick = {1e-2, 1e-1, 1e0, 1e1,1e2,1e3,1e4}]
\addplot[line width=0.7pt,color=blue,mark=triangle]
table[x=npts,y=Ltim]{./ResultsNew/matlabLogger1.txt};
\label{pgfplots:plot1D1}
\addplot[line width=0.7pt,color=darkgreen,mark=square]
table[x=npts,y=Ltim]{./ResultsNew/matlabLogger2.txt};
\label{pgfplots:plot2D1}
\addplot[line width=0.7pt,color=red,mark=o]
table[x=npts,y=Ltim]{./ResultsNew/matlabLogger3.txt};
\label{pgfplots:plot3D1}
\addplot[line width=0.7pt,color=black,dashed]
table[x=npts,y expr={0.7e-6 * x^1.5}]{%
./ResultsNew/matlabLogger1.txt};
\addplot[line width=0.7pt,color=black,dashed]
table[x=npts,y expr={0.1e-6 * x^2}]{%
./ResultsNew/matlabLogger1.txt};
\label{pgfplots:asymps1}
\end{loglogaxis}
\begin{loglogaxis}[%
xshift=0.405\textwidth,width=0.42\textwidth,grid=both,
 ymin= 0, ymax = 4e4, xmin = 2048,
 xmax =1.2e6,ytick={1e1, 1e2, 1e3, 1e4},
 legend style={legend pos=south east,font=\small},
  ylabel={\small $\operatorname{anz}({\bs L})$},
  xlabel ={\small $N$}]
\addplot[line width=0.7pt,color=blue,mark=triangle]%
table[x=npts, y = nzL]{./ResultsNew/matlabLogger1.txt};
\addplot[line width=0.7pt,color=darkgreen,mark=square]%
 table[x=npts, y = nzL]{./ResultsNew/matlabLogger2.txt};
\addplot[line width=0.7pt,color=red,mark=o]%
 table[x=npts, y = nzL]{./ResultsNew/matlabLogger3.txt};
\end{loglogaxis}
\matrix[
    matrix of nodes,
    anchor=north west,
    draw,
    inner sep=0.1em,
    column 1/.style={nodes={anchor=center}},
    column 2/.style={nodes={anchor=west},font=\strut},
    draw
  ]
  at([xshift=0.02\textwidth]current axis.north east){
    \ref{pgfplots:plot1D1}& \(d=1\)\\
    \ref{pgfplots:plot2D1}& \(d=2\)\\
    \ref{pgfplots:plot3D1}& \(d=3\)\\
    \ref{pgfplots:asymps1}& \(N^{\frac{3}{2}}\),
    \(N^2\)\\};
\end{tikzpicture}}
\caption{\label{fig:cholTimesNNZ}Computation times
for the
Cholesky factorization (left) and average numbers
of nonzeros 
per row for the Cholesky factor (right) versus
the number 
sample points $N$ in case of the exponential
kernel matrix.}
\end{center}
\end{figure}

Next, we examine the Cholesky factorization of
the compressed
kernel matrix. As the largest eigenvalue of the
kernel matrix
grows proportionally to the number \(N\) of points,
while the smallest eigenvalue is
given by the ridge parameter, the condition number
grows with \(N\) as well.
Hence, to obtain a constant condition number for
increasing \(N\), the ridge parameter needs to be
adjusted accordingly.
However, as we are only interested in the generated
fill-in and the computation times,
we neglect this fact and just fix the ridge parameter
to \(\rho=1\) for all considered \(N\) and \(d=1,2,3\).
The obtained results are found in 
Figure~\ref{fig:cholTimesNNZ}. Herein, on the left-hand 
side, the wall times for the Cholesky factorization of 
the reordered matrix are found. For \(d=1\),
the average number of nonzeros per row becomes constant
when  the number \(N\) of points increases. This indicates
that the kernel function is already fully resolved up
to the threshold parameter on the coarser levels.
\textcolor{red}{For \(d=2\), the observed rate is slightly worse than
 the expected one of \(N^{\frac{3}{2}}\) for the
 Cholesky factorization, which is caused by the
 high connectivity of the associated graph. Asymptotically,
 the expected reate seems to be achieved. 
 Likewise, for \(d=3\), one figures out the rate 
 \(N^{2.3}\) in contrast to the expected rate \(N^2\). 
 This is again caused by the high connectivity 
 of the associated graph.}
 On the right-hand side of the
same figure, it can be seen that the fill-in
remains rather moderate.
A visualization of the matrix patterns for the matrix
\({\bs K}^\Sigma_\varepsilon+\rho{\bs I}\),
the reordered matrix and the Cholesky factor for
\(N=131\,072\) points is
shown in Figure~\ref{fig:patterns}. Each dot
corresponds to a block of
\(256\times 256\) matrix entries and its intensity
indicates the number
of nonzero entries, where darker blocks contain more
entries than lighter blocks.

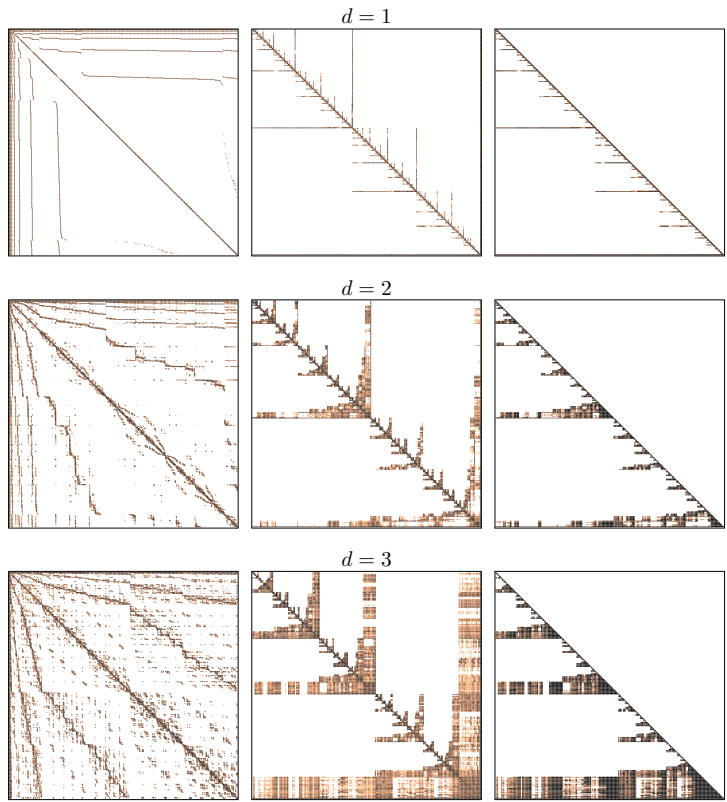
\begin{figure}[htb]
\begin{center}
\scalebox{0.8}{
\begin{tikzpicture}
\draw(0,4.5) node%
 {\includegraphics[scale=0.31,frame,%
 trim= 0 0 0 13.4,clip]{%
 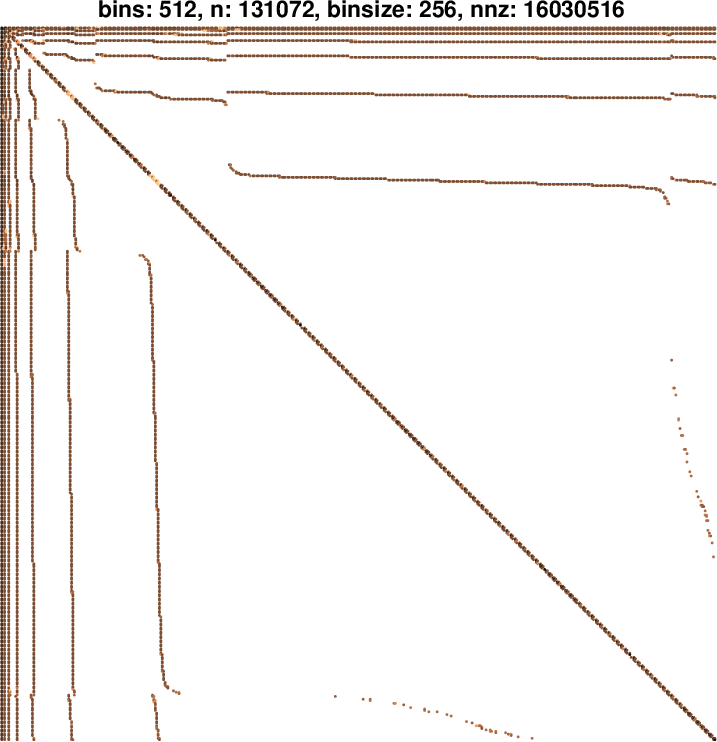}};
\draw(4,4.5) node {\includegraphics[scale=0.31,frame,%
trim= 0 0 0 13.4,clip]{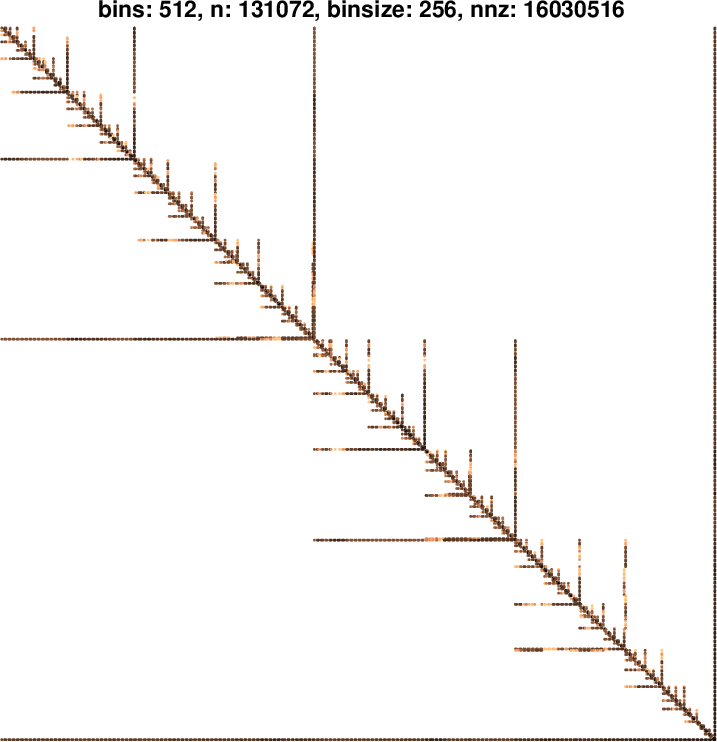}};
\draw(8,4.5) node%
 {\includegraphics[scale=0.31,frame,%
 trim= 0 0 0 13.4,clip]{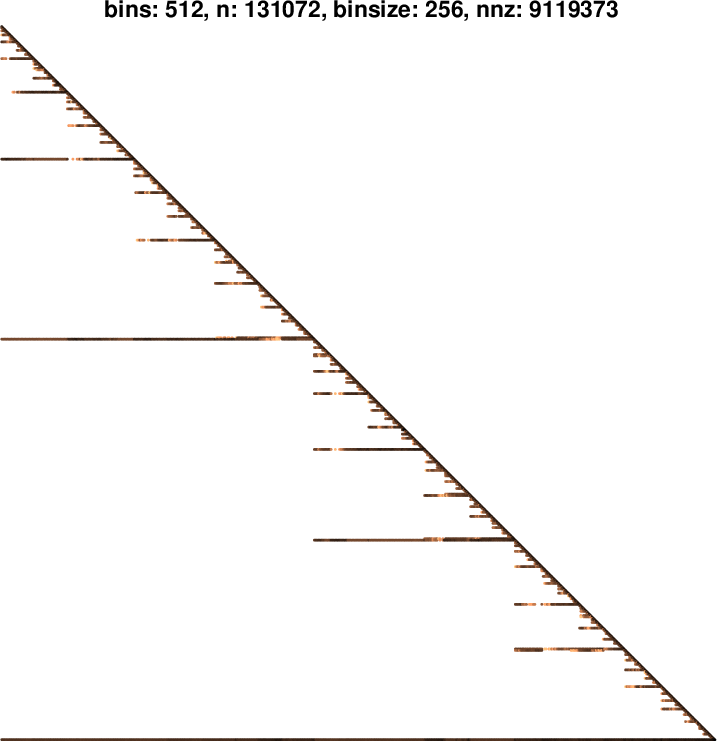}};
\draw(4,6.6) node {$d=1$};
\draw(0,0) node%
 {\includegraphics[scale=0.31,frame,%
 trim=  0 0 0 13.4,clip]{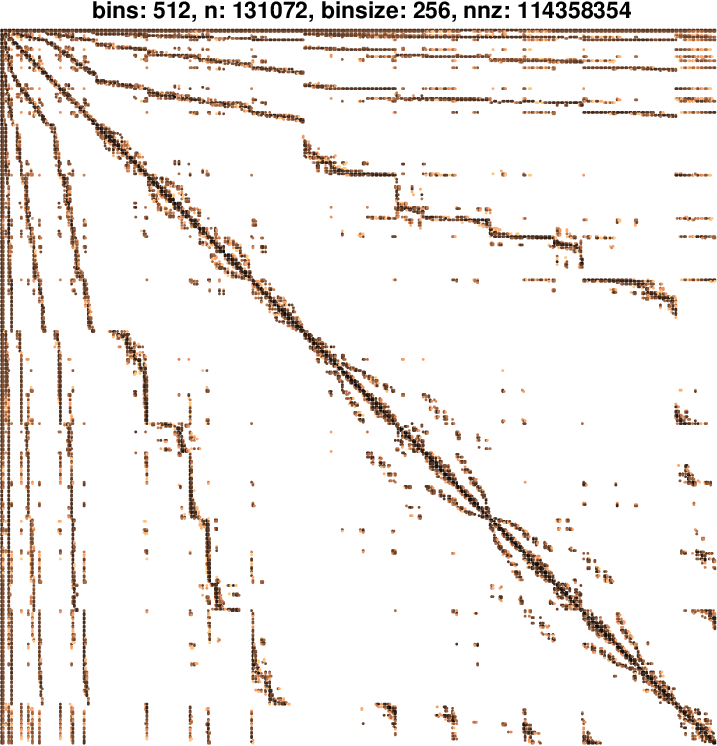}};
\draw(4,0) node%
 {\includegraphics[scale=0.31,frame,%
 trim=  0 0 0 13.4,clip]{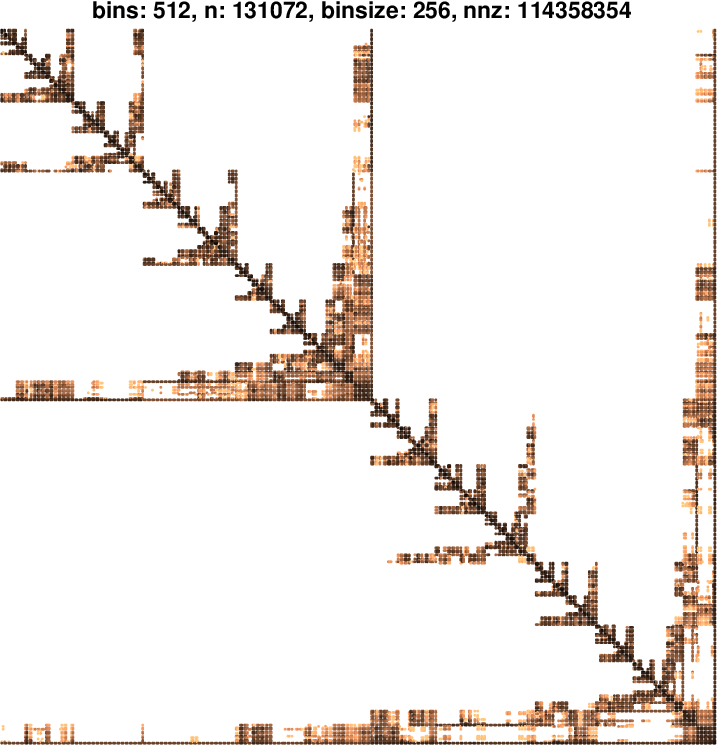}};
\draw(8,0) node%
 {\includegraphics[scale=0.31,frame,%
 trim=   0 0 0 13.4,clip]{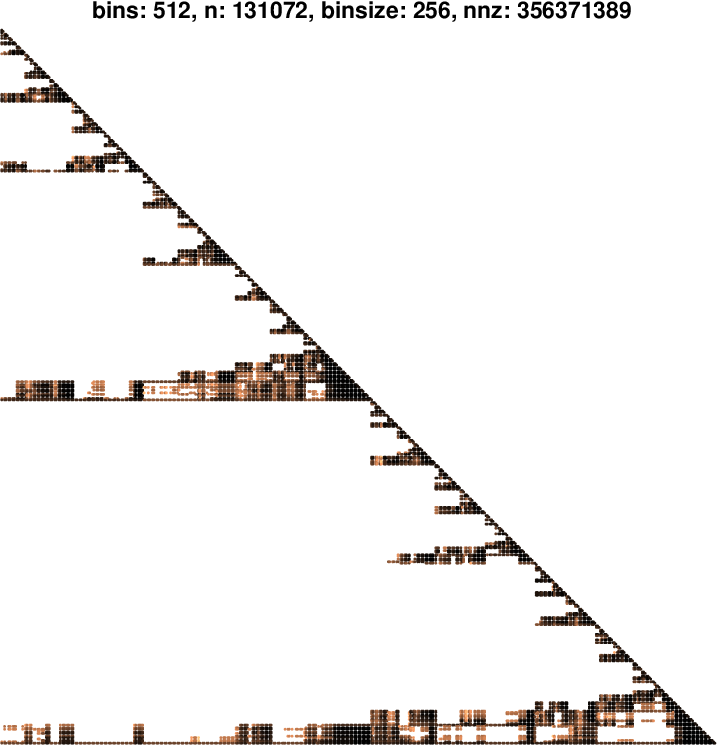}};
\draw(4,2.1) node {$d=2$};
\draw(0,-4.5) node%
 {\includegraphics[scale=0.31,frame,%
 trim= 0 0 0 13.4,clip]{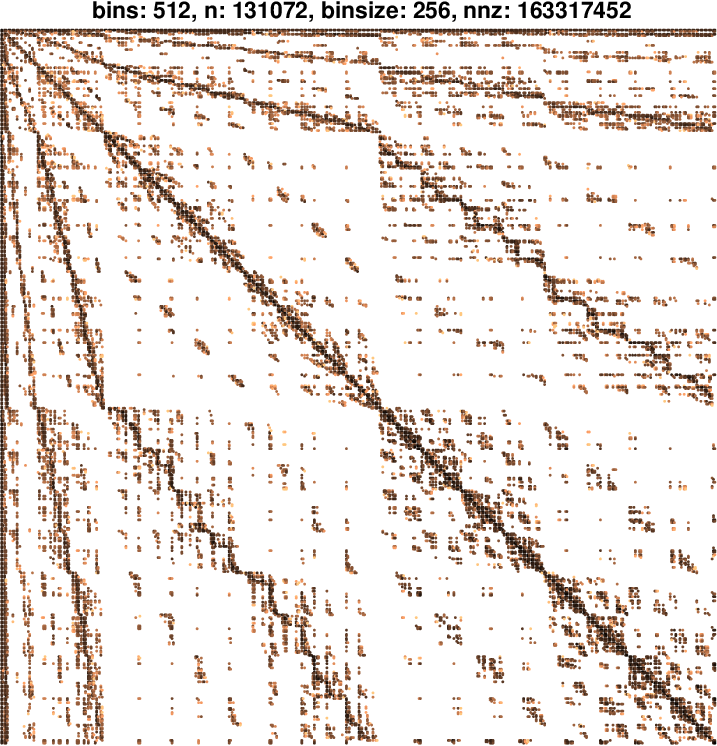}};
\draw(4,-4.5) node%
 {\includegraphics[scale=0.31,frame,%
 trim=  0 0 0 13.4,clip]{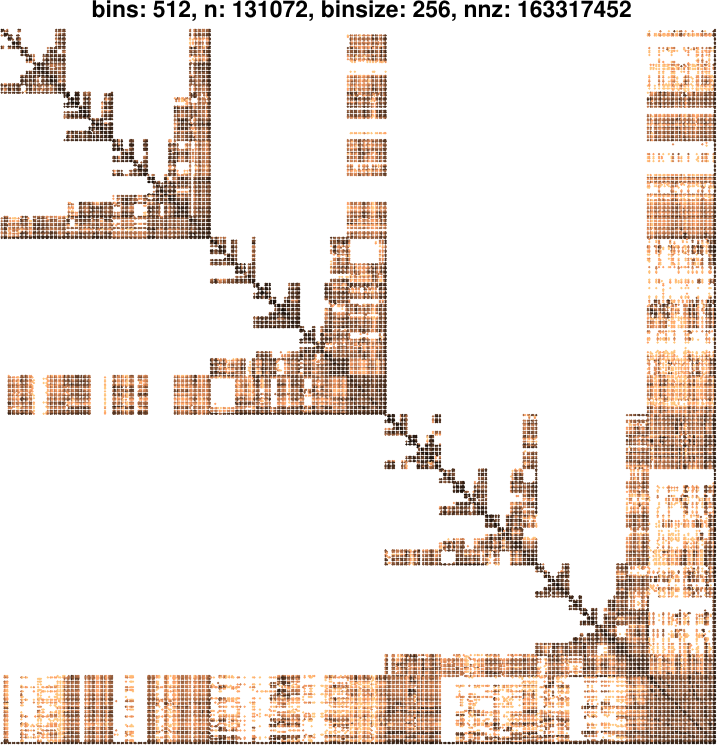}};
\draw(8,-4.5) node%
{\includegraphics[scale=0.31,frame,%
trim=  0 0 0 13.4,clip]{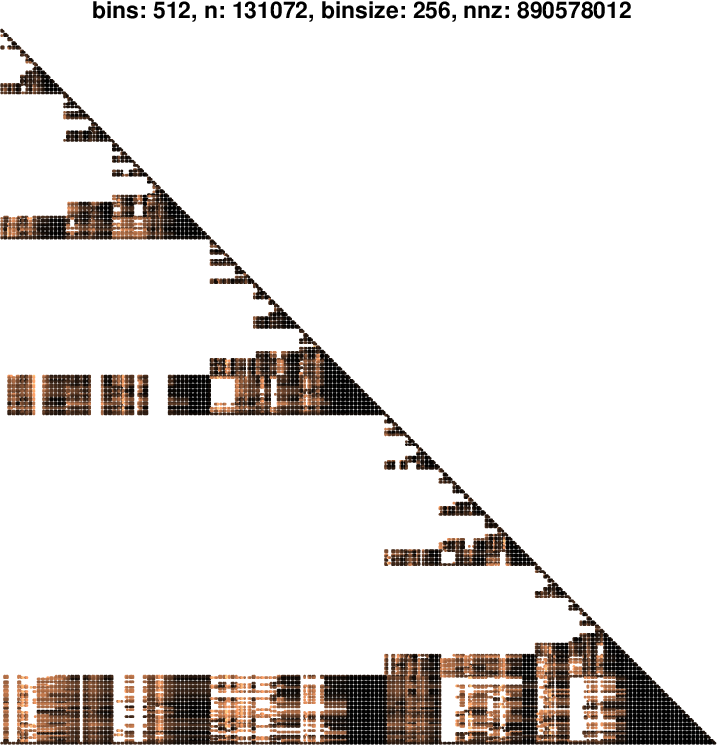}};
\draw(4,-2.4) node {$d=3$};
\end{tikzpicture}}
\caption{\label{fig:patterns}Sparsity pattern of
\({\bs K}^\Sigma_\varepsilon+\rho{\bs I}\) (left),
the reordered matrices (middle) and the Cholesky
 factors \({\bs L}\) (right)
for \(d=1,2,3\) and \(N=131\,072\).}
\end{center}
\end{figure}
\subsection*{Simulation of a Gaussian random field}
As our last example, we consider a Gaussian random
field evaluated at 100\,000 randomly chosen points at
the surface of the Stanford bunny. As before, the
Stanford bunny has been rescaled to have a diameter
of 2.  In order to demonstrate that our approach works
also for larger dimensions, the Stanford bunny has been
embedded into \(\mathbb{R}^4\) and randomly rotated to
prevent axis-aligned bounding boxes. The polynomial
degree for  the \(\Hcal^2\)-matrix representation is
set to 3 as before and likewise we consider \(q+1=3\)
vanishing moments. The covariance function is given by
the exponential kernel 
\[
k({\bs x},{\bs y})=e^{-25\|{\bs x}-{\bs y}\|_2}.
\]
Moreover, we discard all computed matrix entries which
are below the threshold of \(\varepsilon=10^{-6}\). 
The ridge parameter is set to \(\rho=10^{-2}\).
The compressed covariance matrix exhibits
\(\operatorname{anz}({\bs K}^\Sigma_\varepsilon)=6457\)
nonzero matrix entries per row on average, while the
corresponding Cholesky factor exhibits
\(\operatorname{anz}({\bs L})=14\,898\) nonzero
matrix entries per row on average. 
Having the Cholesky factor \({\bs L}\) at hand,
the computation of a realization of the
Gaussian random field is extremely fast, as it only
requires a simple sparse matrix-vector multiplication
of \({\bs L}\) by a Gaussian random vector and an
inverse samplet transform. Four different realizations
of the random field projected
to \(\mathbb{R}^3\) are shown in Figure~\ref{fig:GRF}.

\begin{figure}[htb]
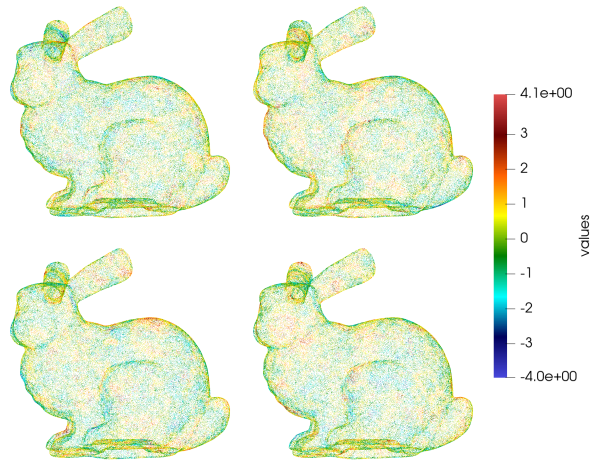

\begin{center}
\scalebox{0.8}{
\begin{tikzpicture}
\draw (0,4) node {\includegraphics[scale=0.08,clip,%
 trim= 840 330 850 400]{./Results/bunnyField1.png}};
\draw (4,4) node {\includegraphics[scale=0.08,clip,%
 trim= 840 330 850 400]{./Results/bunnyField2.png}};
\draw (0,0) node {\includegraphics[scale=0.08,clip,%
 trim= 840 330 850 400]{./Results/bunnyField3.png}};
\draw (4,0) node {\includegraphics[scale=0.08,clip,%
 trim= 840 330 850 400]{./Results/bunnyField4.png}};
\draw (7,2) node {\includegraphics[scale=0.2,clip,%
 trim= 2260 450 450 750]{./Results/bunnyField4.png}};
\end{tikzpicture}}
\caption{\label{fig:GRF}Four different realizations of
a Gaussian random field based on an exponential covariance kernel.}
\end{center}
\end{figure}

\section{Conclusion}\label{sec:Conclusion}
Samplets provide a new methodology for the analysis
of large data sets. They are easy to construct and
discrete  data can be transformed into the samplet
basis in linear cost. In our construction, we
deliberately let out the discussion of a level
dependent compression of the given data, as it 
is known from wavelet analysis, in favor of a robust
error analysis. We emphasize however that, under the
assumption  of uniformly distributed points, different
norms can be incorporated, allowing for the
construction of band-pass filters and level dependent
thresholding. In this situation, also an improved
samplet matrix compression is possible 
such that a fixed number of vanishing moments is
sufficient to achieve a precision proportional to the
fill distance with log-linear cost.

Besides data compression, detection of singularities 
and adaptivity, we have demonstrated how samplets can
be employed for the compression kernel matrices to
obtain an essentially sparse matrix. Having a sparse
representation of the kernel matrix, algebraic
operations, such as matrix vector multiplications
can considerably be sped up. Moreover, in combination 
with a fill-in reducing reordering, the factorization
of the compressed kernel matrices becomes
computationally feasible, which allows for the fast
application of the inverse kernel matrix on the one
hand and the efficient solution of linear systems
involving the kernel matrix on the other hand. The
numerical results, featuring about \(10^6\) data points
in up to four dimensions, demonstrate the capabilities
of samplets.

Future research will be directed to 
the extension of samplets towards high-dimensional
 data.
This extension requires the incorporation of different
clustering strategies, such as locality sensitive
 hashing,
to obtain a manifold-aware cluster tree and the careful
construction for the vanishing moments, for example 
by anisotropic polynomials.

\bibliographystyle{plain}

\end{document}